\documentclass[10pt]{article}

\usepackage{amssymb,latexsym,amsmath,enumerate,verbatim,amsfonts,amsthm}
\usepackage{color} 
\usepackage[pdftex,pdfstartview=FitH]{hyperref}
\usepackage{algorithm}

\hypersetup{colorlinks=false,linkbordercolor={1 1 1},citebordercolor={1 1 1}}

\newtheorem{lemma}{Lemma}[section]
\newtheorem{definition}{Definition}[section]

\newtheorem{assumption}{Assumption}[section]
\newtheorem{theorem}{Theorem}[section]

\newtheorem{proposition}{Proposition}[section]
\newtheorem{remark}{Remark}[section]
\newtheorem{example}{Example}[section]

\def\R{{\rm I\!R}}

\def\Argmin{\mathop{\rm Arg\,min}}

\def\argmin{\mathop{\rm arg\,min}}

\title{\sf Further properties of the forward-backward envelope with applications to difference-of-convex programming}

\author{
Tianxiang Liu\thanks{Department of Applied Mathematics, the Hong Kong Polytechnic University, Hong Kong. This author was supported partly by the AMSS-PolyU Joint Research Institute Postdoctoral Scheme. E-mail: {tiskyliu@polyu.edu.hk}.} 
\and
Ting Kei Pong \thanks{Department of Applied Mathematics, the Hong Kong Polytechnic University, Hong Kong.
This author was supported partly by Hong Kong Research Grants Council PolyU253008/15p. E-mail: {tk.pong@polyu.edu.hk}.}
}

\date{October 15, 2016}   

\begin{document}
\maketitle

\begin{abstract}
  In this paper, we further study the forward-backward envelope first introduced in \cite{FBE14} and \cite{FBE16} for problems whose objective is the sum of a proper closed convex function and a twice continuously differentiable
  possibly nonconvex function with Lipschitz continuous gradient. We derive sufficient conditions on the original problem for the corresponding forward-backward envelope to be a level-bounded and Kurdyka-{\L}ojasiewicz function with an exponent of $\frac12$; these results are important for the efficient minimization of the forward-backward envelope by classical optimization algorithms. In addition, we demonstrate how to minimize some difference-of-convex regularized least squares problems by minimizing a suitably constructed forward-backward envelope. Our preliminary numerical results on randomly generated instances of large-scale $\ell_{1-2}$ regularized least squares problems \cite{YLHX14} illustrate that an implementation of this approach with a limited-memory BFGS scheme usually outperforms standard first-order methods such as the nonmonotone proximal gradient method in \cite{WrNoFi09}.
\end{abstract}

\section{Introduction}

In this paper, we consider the following optimization problem
\begin{equation}\label{P0}
  \min_{x} \ f(x) + P(x),
\end{equation}
where $P:\R^n\to \R\cup\{\infty\}$ is a proper closed convex function, $f:\R^n\to \R$ is twice continuously differentiable and there exists $L > 0$ so that all eigenvalues of $\nabla^2 f(x)$ lie within the closed interval $[-L,L]$, for all $x\in \R^n$. We assume that the proximal mapping of $\tau P$ is easy to compute for any $\tau > 0$; see Section~\ref{sec2} for the definition of proximal mapping. This problem arises in various applications where $f$ is typically a loss function and $P$ is introduced to induce desirable structures in the solution. Concrete examples include many popular models in image processing \cite{Cham04}, as well as compressed sensing \cite{CaTa05,Don06}, where $f$ is the least squares loss function and $P$ is the $\ell_1$ norm or the indicator function of the $\ell_1$ norm ball.

Since these practical problems are often presented in large-scale, a lot of existing methods for solving \eqref{P0} are first-order methods, whose cost per iteration is low, thanks to the fact that the proximal mapping of $\tau P$ is easy to compute for any $\tau > 0$. We refer the readers to the recent exposition \cite{Tseng2010} for an overview of first-order methods such as the proximal gradient algorithm and its several accelerated variants for solving \eqref{P0} when $f$ is, in addition, convex. Despite the prominence of first-order methods, tremendous efforts have also been devoted to developing algorithms that can take into account second-order information to possibly accelerate convergence. A class of such algorithms involves replacing the computation of the proximal mapping in each iteration by the computation of a {\em scaled} proximal mapping. The resulting subproblem in each iteration then needs to be solved by an iterative scheme; see \cite{LeSuSa14} for an example of such algorithms, and \cite{FrGoh16} for efficient evaluation of scaled proximal mappings. Another class of algorithms involves using semi-smooth Newton techniques to solve a nonsmooth equation that characterizes optimality/stationarity; see Section~\ref{sec2} for the definition of stationary points. One example of such equations is
\begin{equation}\label{nonsmootheq}
x - {\rm prox}_{\gamma P}(x - \gamma \nabla f(x)) = 0,
\end{equation}
for some $\gamma > 0$, and we refer the readers to \cite{BCNO15,GrLo08,MiUl14,XLWZ16} for other variants and more detailed discussions. An interesting variant along this line was proposed recently in \cite{FBE14} (see also \cite{FBE16}), which considered a potential function whose gradient is {\em roughly} the left hand side of \eqref{nonsmootheq}. Specifically, the authors introduced the so-called forward-backward envelope for \eqref{P0}. The forward-backward envelope corresponding to \eqref{P0} is defined as
\begin{equation*}
  F_\gamma(x) = f(x) - \frac{\gamma}{2}\|\nabla f(x)\|^2 + P^\gamma(x - \gamma \nabla f(x)),
\end{equation*}
where $P^\gamma(u):= \inf\limits_y\left\{P(y) + \frac{1}{2\gamma}\|y - u\|^2\right\}$ and $\gamma \in (0,\frac1L)$. It was shown in \cite[Proposition~1]{FBE14} and \cite[Theorem~2.6]{FBE16} that $F_\gamma$ is smooth with its gradient given by
\begin{equation*}
\nabla F_\gamma(x) = \gamma^{-1}(I - \gamma \nabla^2 f(x))(x - {\rm prox}_{\gamma P}(x - \gamma \nabla f(x))).
\end{equation*}
In particular, because the proximal mapping of $\gamma P$ is easy to compute by assumption, $\nabla F_\gamma(x)$ will be easy to compute if $\nabla^2 f(x)$ is simple, say, when $f$ is a quadratic.
Comparing the expression of $\nabla F_\gamma$ with \eqref{nonsmootheq}, it is not hard to see that (as was established in \cite[Proposition~1]{FBE14} and \cite[Theorem~2.6]{FBE16}), when $\gamma \in (0,\frac1L)$, $\nabla F_\gamma(x) = 0$ if and only if \eqref{nonsmootheq} holds. Thus, when $\gamma \in (0,\frac1L)$, one can find a stationary point of \eqref{P0} by finding a stationary point of $F_\gamma$. If $F_\gamma$ is level bounded, a stationary point can then be found by many standard and classical optimization methods such as nonlinear conjugate gradient, quasi-Newton methods, etc., with standard (monotone or non-monotone) line-search procedures such as the Armijo rule, thanks to the smoothness of $F_\gamma$. In \cite{FBE14}, based on $F_\gamma$, the authors proposed a (semi-smooth) Newton-type method with standard line-search procedures for minimizing $F_\gamma$ when $f$ is convex and $\gamma \in (0,\frac1L)$; very recently, in \cite{FBE16}, the authors proposed a special algorithm for solving \eqref{P0} based on $F_\gamma$ which alternates between a proximal gradient step and a descent direction, and established convergence under suitable assumptions. The numerical performances of the above approaches were promising.

In this paper, unlike \cite{FBE16} that designs a special algorithm for solving \eqref{P0}, we further study properties of $F_\gamma$ for the efficient application of classical/standard optimization algorithms. Specifically, we discuss sufficient conditions on \eqref{P0} that guarantee $F_\gamma$ to be a level-bounded and Kurdyka-{\L}ojasiewicz (KL) function with an exponent of $\frac12$; the latter property is crucial for establishing local linear convergence of a wide variety of standard optimization methods for minimizing $F_\gamma$. Furthermore, we illustrate how some difference-of-convex regularized least squares problems can be reformulated in the form of \eqref{P0} and give conditions on the regularizer so that the corresponding $F_\gamma$ is level bounded. Consequently, various techniques for smooth unconstrained optimization including limited-memory BFGS can now be adapted to solve a class of difference-of-convex regularized least squares problems. Our numerical results show that minimizing a suitable $F_\gamma$ using the limited-memory BFGS direction with Armijo line-search usually outperforms standard first-order methods when solving large-scale $\ell_{1-2}$ regularized least squares problems studied in \cite{YLHX14} with simulated data.

The rest of the paper is organized as follows. In Section~\ref{sec2}, we introduce notation and review some basic properties of $F_\gamma$ of \eqref{P0} discussed in \cite{FBE14} and \cite{FBE16}. Level-boundedness and the KL property of the envelope are studied in Section~\ref{sec:FBE}, and the adaptation of the theory to minimize difference-of-convex regularized least squares problems is presented in Section~\ref{sec:DC}. Numerical results are presented in Section~\ref{sec:num}. Finally, we give some concluding remarks in Section~\ref{sec:conclusion}.

\section{Notation and preliminaries}\label{sec2}

We use $\R^n$ to denote the $n$-dimensional Euclidean space. The standard inner product is denoted by $\langle \cdot,\cdot \rangle$, and the induced norm is denoted by $\|\cdot\|$. In addition, for any $x\in \R^n$, its $i$th entry is denoted by $x_i$, the number of nonzero entries is denoted by $\|x\|_0$, and the $p$-norm, $p \in [1,\infty)$, is denoted by $\|x\|_p:= \left(\sum_{i=1}^n|x_i|^p\right)^{\frac1p}$. We also let $B(x,r)$ denote the closed ball centered at $x$ with radius $r > 0$, and let $U(x,r)$ denote its interior. For a nonempty closed set $\Omega$, we let ${\rm dist}(x,\Omega) := \inf_{y\in \Omega}\|x - y\|$ and use ${\rm Proj}_\Omega(x)$ to denote the set of points in $\Omega$ that are closest to $x$. Such a set reduces to a singleton set when $\Omega$ is convex. Furthermore, we use $\delta_\Omega$ to denote the indicator function of the closed set $\Omega$, i.e.,
\[
\delta_\Omega(x) = \begin{cases}
  0 & {\rm if}\ x\in \Omega,\\
  \infty & {\rm otherwise}.
\end{cases}
\]
For an extended-real-valued function $g:\R^n\to[-\infty,\infty]$, we denote its domain by ${\rm dom}\,g := \{x\in \R^n:\; g(x) < \infty\}$. Such a function is called proper if ${\rm dom}\,g\neq \emptyset$ and $g$ is never $-\infty$. A proper function $g$ is said to be closed if it is lower semicontinuous, and is said to be convex if its epigraph, ${\rm epi}\,g:= \{(x,r)\in \R^n\times \R:\; r\ge g(x)\}$, is a convex set.

We say that a proper closed function $g$ is level bounded if $\{x\in \R^n:\; g(x)\le \gamma\}$ is bounded for all $\gamma \in \R$,\footnote{This is equivalent to $\liminf_{\|x\|\to\infty}g(x) = \infty$; see \cite[Page~83]{AuT03}.} and is coercive if
\[
\liminf_{\|x\|\to\infty}\frac{g(x)}{\|x\|} > 0.
\]
Coercive functions are clearly level bounded. Moreover, it is known that for a proper closed convex function, it is level bounded if and only if it is coercive; see, for example, \cite[Proposition~3.1.3]{AuT03}. For a proper closed convex function $g$,
the subdifferential of $g$ at any $x\in {\rm dom}\,g$ is defined as
\[
\partial g(x) := \{u\in \R^n:\; g(y) - g(x)\ge \langle u,y-x\rangle\ \ \forall y\in \R^n\},
\]
and its proximal mapping is defined as
\[
{\rm prox}_g(x) := \argmin_y \left\{g(y) + \frac12 \|y - x\|^2\right\}.
\]
It is well known that the proximal mapping is a Lipschitz continuous single-valued mapping with a Lipschitz constant $1$, and that
\begin{equation}\label{prox_equi}
  u = {\rm prox}_g(x)\ \ \Longleftrightarrow \ \ x - u \in \partial g(u).
\end{equation}
Finally, we define ${\rm dom}\,\partial g:=\{x\in \R^n:\; \partial g(x)\neq \emptyset\}$.

For problems of the form \eqref{P0}, it is known that any local minimizer $x^*$ has to satisfy
\begin{equation}\label{sub_rule}
0 \in \nabla f(x^*) + \partial P(x^*);
\end{equation}
see \cite[Theorem~10.1]{Rock98} and \cite[Exercise~8.8]{Rock98}. Using \eqref{prox_equi}, it is not hard to see that \eqref{sub_rule} is equivalent to
\begin{equation}\label{sub_rule2}
  x^* = {\rm prox}_{\gamma P}(x^* - \gamma \nabla f(x^*))
\end{equation}
for any $\gamma > 0$. We say that $\bar x$ is a stationary point of \eqref{P0} if \eqref{sub_rule} (or equivalently \eqref{sub_rule2} for some $\gamma > 0$) is satisfied with $\bar x$ in place of $x^*$.
The set of stationary points of \eqref{P0} is denoted by ${\cal X}$. One can show that if $f+P$ is coercive, then ${\cal X}\neq \emptyset$.

For the class of problem \eqref{P0}, a brilliant way of solving it was proposed in \cite{FBE14} and \cite{FBE16}, which is based on the following function called the forward-backward envelope:
\begin{equation*}
\begin{split}
  F_\gamma(x) &:= \inf_y \left\{f(x) + \langle\nabla f(x),y-x\rangle + \frac{1}{2\gamma}\|y - x\|^2 + P(y)\right\}\\
  & = f(x) - \frac{\gamma}{2}\|\nabla f(x)\|^2 + P^\gamma(x - \gamma \nabla f(x)),
\end{split}
\end{equation*}
where $P^\gamma(u):= \inf\limits_y\left\{P(y) + \frac{1}{2\gamma}\|y - u\|^2\right\}$ and $\gamma \in (0,\frac1L)$. This function can be interpreted as a generalized Moreau envelope with a suitable Bregman distance. Indeed, recall that for a differentiable convex function $\phi$, the Bregman distance $D_\phi(y,x)$ is defined as
\[
D_\phi(y,x) := \phi(y) - \phi(x) - \langle\nabla \phi(x),y-x\rangle.
\]
If we take $\phi(x) = \frac{1}{2\gamma}\|x\|^2 - f(x)$, then $\phi$ is a convex differentiable function. Moreover, it is not hard to show from the definitions that
\begin{equation}\label{FBE2}
F_\gamma(x) = \inf_y \left\{f(y) + P(y) + D_\phi(y,x)\right\}.
\end{equation}
Hence, $F_\gamma$ is just a generalized Moreau envelope that uses a suitable Bregman distance in place of the square of Euclidean distance. We refer the readers to \cite{BBC03,KS12} for more details on the generalized Moreau envelope and its properties.

It was shown in \cite[Proposition~1]{FBE14}, \cite[Proposition~2.3]{FBE16} and \cite[Theorem~2.6]{FBE16} that $F_\gamma$ enjoys the following nice properties:
\begin{enumerate}[{\rm (i)}]
  \item $F_\gamma$ is continuously differentiable for any $\gamma \in (0,\frac1L)$, with its gradient given by
\begin{equation}\label{gradFgamma}
\nabla F_\gamma(x) = \gamma^{-1}(I - \gamma \nabla^2 f(x))(x - {\rm prox}_{\gamma P}(x - \gamma \nabla f(x))).
\end{equation}
  \item When $\gamma \in (0,\frac1L)$, the set of stationary points of $F_\gamma$ equals $\cal X$, and the set of global minimizers of $F_\gamma$ equals that of $f+P$.
\end{enumerate}
Consequently, in order to minimize \eqref{P0}, as suggested in \cite{FBE14} and \cite{FBE16}, it is natural to consider the following possibly nonconvex problem for a fixed $\gamma \in (0,\frac1L)$:
\begin{equation}\label{P3}
\min_x \ \ F_\gamma(x),
\end{equation}
which is an unconstrained optimization problem with a smooth objective and is potentially solvable by many classical optimization algorithms such as quasi-Newton methods with standard line-search strategies. As discussed in the introduction of \cite{FBE14} and \cite{FBE16}, such an approach is in spirit similar to the merit function approach for solving variational inequality problems. We will discuss more properties of $F_\gamma$ in the next section.

Before ending this section, we state the following Kurdyka-{\L}ojasiewicz (KL) property, which is a useful property for establishing convergence rate of optimization algorithms; see, for example, \cite{AB09,Jerome,LiP16}. This property will be further studied in the next section. Its definition given below is adapted from \cite[Definition~2.3]{LiP16}.
\begin{definition}\label{KL}{\bf (Kurdyka-{\L}ojasiewicz property for smooth functions)}
  For $\theta \in (0,1)$, we say that a smooth function $g$ is a Kurdyka-{\L}ojasiewicz (KL) function with an exponent of $\theta$ if for any $\bar x\in \R^n$, there exist $c$ and $\epsilon > 0$ so that
  \[
  c(g(x) - g(\bar x))^\theta\le \|\nabla g(x)\|
  \]
  whenever $\|x-\bar x\|< \epsilon$ and $g(\bar x)\le g(x)$.
\end{definition}

\section{Further properties of the forward-backward envelope}\label{sec:FBE}

In this section, we further study properties of the forward-backward envelope $F_\gamma$ of \eqref{P0} beyond those already known in \cite{FBE14} and \cite{FBE16}. Specifically, we show that under certain mild conditions on the original problem \eqref{P0}, the function $F_\gamma$ is level bounded, which would guarantee boundedness of the sequence generated by standard descent methods for solving \eqref{P3}. In addition, we also establish that $F_\gamma$ has the KL property under suitable assumptions on the original problem \eqref{P0}. These results will then be applied to establishing convergence of some classical optimization algorithms applied to minimizing $F_\gamma$.

We start by giving a sufficient condition for the level-boundedness of $F_\gamma$.

\begin{theorem}\label{thm1}{\bf (Level-boundedness of $F_\gamma$)}
  If $f+P$ is coercive and $\gamma \in (0,\frac1L)$, then $F_\gamma$ is level bounded.
\end{theorem}
\begin{remark}
  Recall from \cite[Proposition~3.1.3]{AuT03} that a proper closed convex function is level bounded if and only if it is coercive. Thus, when $f$ is in addition convex, we see from Theorem~\ref{thm1} that, if $f+P$ is level bounded and $\gamma \in (0,\frac1L)$, then $F_\gamma$ is level bounded.
\end{remark}
\begin{proof}
  First, notice that $f$ has a Lipschitz continuous gradient with a Lipschitz constant of $L > 0$. From this, it is not hard to show that for any $x$ and $y$,
  \begin{equation*}
    D_\phi(y,x)\ge \frac12\left(\frac{1}{\gamma} - L\right)\|y - x\|^2
  \end{equation*}
  when $\phi(\cdot) = \frac{1}{2\gamma}\|\cdot\|^2 - f(\cdot)$ and $\gamma \in (0,\frac1L)$.
  Combining this with \eqref{FBE2}, we see further that
  \begin{equation}\label{fbound2}
    \begin{split}
      F_\gamma(x)& = \inf_y \left\{f(y) + P(y) + D_\phi(y,x)\right\} \ge \inf_y \left\{f(y) + P(y) + \frac12\left(\frac{1}{\gamma} - L\right)\|y - x\|^2\right\}.
    \end{split}
  \end{equation}
  Next, from the definition of coerciveness, there are $r > 0$ and $\alpha > 0$ so that
  \[
  (f + P)(x) > \frac\alpha{2} \|x\|
  \]
  whenever $\|x\| > r$. Moreover, since $f + P$ is proper lower semicontinuous, it follows that there exists $\beta > -\infty$ so that
  \[
  \inf_{\|x\| \le r} (f + P)(x) = \beta.
  \]
  These together with \eqref{fbound2} give further that
  \begin{equation*}
  \begin{split}
    F_\gamma(x) &\ge \min\left\{\inf_{\|y\|\le r}\left\{\beta + \frac12\left(\frac{1}{\gamma} - L\right)\|y - x\|^2\right\},
    \inf_{\|y\|> r}\left\{\frac{\alpha}{2}\|y\| + \frac12\left(\frac{1}{\gamma} - L\right)\|y - x\|^2\right\}\right\}\\
    & = \min\left\{\beta + \frac12\left(\frac{1}{\gamma} - L\right)(\|x\| - r)_+^2,G(x)\right\},
  \end{split}
  \end{equation*}
  where
  \[
  G(x) = \begin{cases}
    \frac{\alpha r}{2} + \frac12\left(\frac1\gamma - L\right)(\|x\|-r)^2 & {\rm if}\ \|x\| - \frac{\alpha}{2\left(\frac1\gamma-L\right)} \le r,\\
    \frac\alpha2 \|x\| - \frac{\alpha^2}{8\left(\frac1\gamma-L\right)} & {\rm else}.
  \end{cases}
  \]
  From these one easily obtains $\liminf_{\|x\|\to\infty}F_\gamma(x) = \infty$. This completes the proof.
\end{proof}

We next derive conditions under which $F_\gamma$ is a KL function. The KL property is useful for establishing convergence of the whole sequence generated by some classical descent algorithms for minimizing $F_\gamma$; see, for example, \cite{AB09,Jerome,AtBoSv13}. In addition, it is known that the so-called KL exponent (the $\theta$ in Definition~\ref{KL}) is closely related to the local convergence rate of these algorithms; see, for example, \cite{AB09,Jerome,LiP16} and references therein. Specifically, an exponent $\theta\in (0,\frac12]$ typically implies local linear convergence of the sequence of function values generated to a stationary value. Hence, it will be useful to derive sufficient conditions that guarantee $F_\gamma$ to be a KL function with an exponent of $\frac12$. We achieve this by considering the following error bound condition, which was first introduced in a series of papers \cite{Luo1992,Luo1993,Tseng1993} for establishing local linear convergence of the matrix splitting algorithms and some classical feasible descent algorithms.

\begin{assumption}\label{assum1}{\bf (Error bound condition)}
Suppose that ${\cal X}\neq \emptyset$, and that for any $\xi \ge \inf (f + P)$, there exist $\epsilon > 0$ and $\tau > 0$ so that
    \[
    {\rm dist}(x,{\cal X}) \le \tau \|x - {\rm prox}_P(x - \nabla f(x))\|
    \]
    whenever $\|x - {\rm prox}_P(x - \nabla f(x))\| < \epsilon$ and $f(x) + P(x)\le \xi$.
\end{assumption}

It is known that if $\inf(f+P)>-\infty$ and the error bound condition holds together with a condition concerning the separation of stationary values, then many first-order methods applied to solving \eqref{P0} can be shown to be locally linearly convergent; see, for example, \cite{Luo1993,TseY09}. In addition, when $\cal X\neq \emptyset$, this error bound condition is satisfied by a large class of functions that arise in practical applications; for example, it is satisfied when $f$ is a (not necessarily convex) quadratic function and $P$ is a proper polyhedral function; see \cite[Theorem~4]{TseY09}. We refer the readers to \cite{Luo1992,Luo1993,Tseng1993,TseY08,TseY09,Tseng2010} and references therein for more examples.

In the next theorem, we show that if the error bound condition holds, the function $f$ is analytic and the function $P$ is continuous on ${\rm dom}\,\partial P$, bounded below and subanalytic (see, for example \cite[Definition~2.1]{BDL06} and \cite[Definition~6.6.1]{Finite_VI1}), then $F_\gamma$ is a KL function with an exponent of $\frac12$; see Theorem~\ref{thm2} (b) below. Using this latter property together with some standard assumptions, many first-order methods applied to solving \eqref{P3} can be shown to be locally linearly convergent, as we will also demonstrate later in Proposition~\ref{convergence}.

We need the following auxiliary lemma for proving our theorem. 
\begin{lemma}\label{lem1}
  Let $\gamma \in (0,\frac1L)$ and define
  \begin{equation*}
  {\cal P}_\gamma(x) = {\rm prox}_{\gamma P}\left(x - \gamma \nabla f(x)\right).
  \end{equation*}
  Then for any $x$ and $y$, we have
  \[
  \|{\cal P}_\gamma(x) - {\cal P}_\gamma(y)\| \le 2\|x - y\|.
  \]
  In particular, if $\bar x\in \cal X$, then $\|{\cal P}_\gamma(x) - \bar x\| \le 2\|x - \bar x\|$.
\end{lemma}
\begin{proof}
  For any $x$ and $y$, we have
  \[
  \|{\cal P}_\gamma(x) - {\cal P}_\gamma(y)\|\le \left\|\left(x - \gamma \nabla f(x)\right) - \left(y - \gamma \nabla f(y)\right)\right\| \le 2\|x - y\|,
  \]
  where the first inequality is a consequence of the nonexpansiveness of proximal mappings (see, for example, \cite[Proposition~12.27]{BauCom11}), and the second inequality follows from the fact that $x \mapsto x - \gamma \nabla f(x)$ is Lipschitz continuous with a Lipschitz constant of $2$. This proves the first inequality. The conclusion concerning stationary points follows from this and the fact that ${\cal P}_\gamma(\bar x) = \bar x$ for any stationary point $\bar x\in \cal X$. 
\end{proof}

\begin{theorem}{\bf (Kurdyka-{\L}ojasiewicz property of $F_\gamma$)}\label{thm2}
  Suppose that the function $f$ is analytic, and $P$ is continuous on ${\rm dom}\,\partial P$, subanalytic and bounded below (i.e., $\inf P > -\infty$).
  Then the following statements hold for any $\gamma \in (0,\frac1L)$ and $\bar x\in \R^n$.
  \begin{enumerate}[{\rm (a)}]
    \item There exist an exponent $\theta\in(0,1)$ and $c$, $\epsilon > 0$ so that
  \begin{equation}\label{eq}
  c|F_\gamma(x) - F_\gamma(\bar x)|^\theta \le \|\nabla F_\gamma(x)\|
  \end{equation}
  whenever $\|x - \bar x\| < \epsilon$.
    \item Suppose in addition that Assumption~\ref{assum1} holds. Then there exist $c$, $\epsilon > 0$ so that
  \begin{equation}\label{eq1}
  c\sqrt{F_\gamma(x) - F_\gamma(\bar x)} \le \|\nabla F_\gamma(x)\|
  \end{equation}
  whenever $\|x - \bar x\| < \epsilon$ and $F_\gamma(\bar x) \le F_\gamma(x)$.
  \end{enumerate}
\end{theorem}
\begin{remark}
  Inequality \eqref{eq} is known as the {\L}ojasiewicz inequality, originally studied for analytic functions; see \cite[Page~598]{Finite_VI1} and references therein for more discussions. It differs slightly from the inequality in Definition~\ref{KL} in the sense that it does not exclude those $x$ satisfying $F_\gamma(x) < F_\gamma(\bar x)$, and considers $|F_\gamma(x) - F_\gamma(\bar x)|$ on the left hand side.
\end{remark}
\begin{proof}
  The conclusions of the proposition hold trivially at any $\bar x$ with $\|\nabla F_\gamma(\bar x)\|\neq 0$ because the left hand side of the inequalities \eqref{eq} and \eqref{eq1} are zero at $\bar x$ and $F_\gamma$ is continuously differentiable. Thus, without loss of generality, we assume that $\|\nabla F_\gamma(\bar x)\| = 0$, i.e., $\bar x\in \cal X$.

  We first prove (a). We start by showing that $F_\gamma$ is subanalytic. To this end, note that $f(x) - \frac{\gamma}{2}\|\nabla f(x)\|^2$ is clearly analytic, and thus subanalytic. In addition, we have from \cite[Proposition~2.9]{BDL06} that $P^\gamma$ is subanalytic and continuous. Since the composition of two continuous subanalytic functions is subanalytic \cite[Page~597~(p5)]{Finite_VI1}, we see that $x\mapsto P^\gamma(x - \gamma \nabla f(x))$ is subanalytic. Finally, since the sum of continuous subanalytic functions is subanalytic \cite[Page~597~(p5)]{Finite_VI1}, we conclude that $F_\gamma$ is subanalytic.

  Since $F_\gamma$ is subanalytic and continuous, by \cite[Theorem~3.1]{BDL06}, there exist an exponent $\theta \in (0,1)$ and $c$, $\epsilon > 0$ so that
  \[
  c|F_\gamma(u) - F_\gamma(\bar x)|^\theta \le \|\nabla F_\gamma(u)\|
  \]
  whenever $\|u  - \bar x\| < \epsilon$. This proves (a).

  We now prove (b). First, from \eqref{eq}, it is not hard to see that for any $u$ with $\|u - \bar x\| < \epsilon$, if $u\in {\cal X}$, then $F_\gamma(u) = F_\gamma(\bar x)$. Now, consider any $u$ with $\|u - \bar x\| < \frac{\epsilon}{2}$ and any $y\in {\rm Proj}_{\cal X}(u)$. Then
  \begin{equation}\label{bd2}
  \|y - \bar x\| \le \|y - u\| + \|u - \bar x\| = {\rm dist}(u,{\cal X}) + \|u - \bar x\| \le 2\|u - \bar x\| < \epsilon.
  \end{equation}
  Thus, it follows that $F_\gamma(y) = F_\gamma(\bar x)$. Furthermore, it is routine to show that $\nabla F_\gamma$ is Lipschitz continuous on $B(\bar x,\epsilon)$ using the formula of $\nabla F_\gamma$ and the analyticity of $f$. We denote a Lipschitz constant of $\nabla F_\gamma$ by $L_f$. Then it holds that whenever $\|x - \bar x\| < \frac{\epsilon}{4}$ so that $\|{\cal P}_\gamma(x) - \bar x\| < \frac{\epsilon}2$ (see Lemma~\ref{lem1}) and hence ${\rm Proj}_{\cal X}({\cal P}_\gamma(x))\subseteq U(\bar x,\epsilon)$ (see \eqref{bd2}), we have
  \begin{equation}\label{upperbound}
  \begin{split}
    F_\gamma({\cal P}_\gamma(x)) - F_\gamma(\bar x) & = F_\gamma({\cal P}_\gamma(x)) - F_\gamma(y) \le \langle \nabla F_\gamma(y),{\cal P}_\gamma(x) - y\rangle + \frac{L_f}2\|{\cal P}_\gamma(x) - y\|^2\\
    & = \frac{L_f}2\|{\cal P}_\gamma(x) - y\|^2 = \frac{L_f}2[{\rm dist}({\cal P}_\gamma(x),{\cal X})]^2,
  \end{split}
  \end{equation}
  where $y\in {\rm Proj}_{\cal X}({\cal P}_\gamma(x))\subseteq U(\bar x,\epsilon)$ and the second equality holds because $\nabla F_\gamma(y) = 0$ for $y \in {\cal X}$. In addition, we also have, whenever $\|x - \bar x\| < \frac{\epsilon}{4}$, that
  \begin{equation}\label{lowerbound}
    \begin{split}
      &F_\gamma(x) - F_\gamma({\cal P}_\gamma(x))  \le \langle \nabla F_\gamma({\cal P}_\gamma(x)),x - {\cal P}_\gamma(x)\rangle + \frac{L_f}{2}\|x - {\cal P}_\gamma(x)\|^2\\
      & = \langle\nabla F_\gamma({\cal P}_\gamma(x)) - \nabla F_\gamma(x),x - {\cal P}_\gamma(x)\rangle + \langle\nabla F_\gamma(x),x - {\cal P}_\gamma(x)\rangle + \frac{L_f}{2}\|x - {\cal P}_\gamma(x)\|^2\\
      & \le L_f\|x - {\cal P}_\gamma(x)\|^2 + \langle\nabla F_\gamma(x),x - {\cal P}_\gamma(x)\rangle + \frac{L_f}{2}\|x - {\cal P}_\gamma(x)\|^2\\
      & \le \frac12\|\nabla F_\gamma(x)\|^2 + \frac{3L_f + 1}{2}\|x - {\cal P}_\gamma(x)\|^2,
    \end{split}
  \end{equation}
  where the last inequality follows from the relation $2\langle\nabla F_\gamma(x),x - {\cal P}_\gamma(x)\rangle \le \|\nabla F_\gamma(x)\|^2 + \|x - {\cal P}_\gamma(x)\|^2$.

  Next, let $\xi := f(\bar x) + P(\bar x) + 1$, and recall from \cite[Lemma~3]{TseY09} that there exist $c_1$ and $c_2 > 0$ so that
  \begin{equation}\label{ineq0}
    c_1\|x - {\rm prox}_P(x - \nabla f(x))\| \le \|x - {\cal P}_\gamma(x)\| \le c_2\|x - {\rm prox}_P(x - \nabla f(x))\|
  \end{equation}
  for all $x$. Combining this with our Assumption~\ref{assum1}, we see that there exist $\eta$ and $\tau > 0$ so that
  \begin{equation}\label{bd}
  {\rm dist}(u,{\cal X}) \le \tau \|u - {\cal P}_\gamma(u)\|
  \end{equation}
  whenever $\|u - {\cal P}_\gamma(u)\| < \eta$ and $f(u) + P(u) \le \xi$. Since both $u\mapsto \|u - {\cal P}_\gamma(u)\|$ and $u\mapsto f(u) + P(u)$ are continuous on ${\rm dom}\,\partial P$, and $\bar x = {\cal P}_\gamma(\bar x)$ due to $\bar x\in \cal X$, by choosing a sufficiently small $\epsilon'\in (0,\frac{\epsilon}4)$, we conclude that \eqref{bd} holds whenever $\|u - \bar x\| < 2\epsilon'$ and $u\in {\rm dom}\,\partial P$.

  Now, consider any $x$ with $\|x - \bar x\| < \epsilon'$. This means $\|{\cal P}_\gamma(x) - \bar x\| < 2\epsilon'$ thanks to Lemma~\ref{lem1}, and hence we have by applying \eqref{bd} with $u = {\cal P}_\gamma(x)\in {\rm dom}\,\partial P$ that
  \begin{equation}\label{bd3}
    {\rm dist}({\cal P}_\gamma(x),{\cal X}) \le \tau \|{\cal P}_\gamma(x) - {\cal P}_\gamma({\cal P}_\gamma(x))\|\le 2\tau\|x - {\cal P}_\gamma(x)\|,
  \end{equation}
  where the last inequality follows from Lemma~\ref{lem1}.

  Finally, note from the formula of $\nabla F_\gamma(x)$ that $\|x - {\cal P}_\gamma(x)\| \le c_0\|\nabla F_\gamma(x)\|$ for some $c_0 > 0$, thanks to the assumption that $\gamma \in (0,\frac1L)$. Using this fact, together with \eqref{upperbound}, \eqref{lowerbound} and \eqref{bd3}, we see further that whenever $\|x - \bar x\| < \epsilon' < \frac\epsilon{4}$,
  \begin{equation*}
    \begin{split}
      & F_\gamma(x) - F_\gamma(\bar x) = F_\gamma({\cal P}_\gamma(x)) - F_\gamma(\bar x) + F_\gamma(x) - F_\gamma({\cal P}_\gamma(x))\\
      & \le \frac{L_f}2[{\rm dist}({\cal P}_\gamma(x),{\cal X})]^2 + \frac12\|\nabla F_\gamma(x)\|^2 + \frac{3L_f + 1}{2}\|x - {\cal P}_\gamma(x)\|^2\\
      & \le 2L_f\tau^2\|x - {\cal P}_\gamma(x)\|^2 + \frac12\|\nabla F_\gamma(x)\|^2 + \frac{3L_f + 1}{2}\|x - {\cal P}_\gamma(x)\|^2 \le C\|\nabla F_\gamma(x)\|^2
    \end{split}
  \end{equation*}
  for some $C > 0$. This completes the proof.
\end{proof}

We next prove a converse to Theorem~\ref{thm2} (b) under suitable conditions, which is of independent interest. Our proof is largely inspired by recent convergence analysis of first-order methods based on the KL property; see, for example, \cite{AtBoSv13}.

\begin{theorem}
  Suppose that $f+P$ is level-bounded, ${\cal X} = \Argmin(f+P)$ and $\gamma \in (0,\frac1L)$. If for any $\bar x \in {\cal X}$, there exist $c$ and $\epsilon > 0$ so that
  \begin{equation}\label{eq2}
    c\sqrt{F_\gamma(x) - F_\gamma(\bar x)} \le \|\nabla F_\gamma(x)\|
  \end{equation}
  whenever $\|x-\bar x\|< \epsilon$ and $F_\gamma(\bar x) \le F_\gamma(x)$, then the error bound condition holds, i.e., ${\cal X}\neq \emptyset$ and for any $\xi \ge \inf (f + P)$, there exist $\epsilon > 0$ and $\tau > 0$ so that
    \[
    {\rm dist}(x,{\cal X}) \le \tau \|x - {\rm prox}_P(x - \nabla f(x))\|
    \]
    whenever $\|x - {\rm prox}_P(x - \nabla f(x))\| < \epsilon$ and $f(x) + P(x)\le \xi$.
\end{theorem}
\begin{proof}
  Since $f+P$ is level-bounded, it holds that $\Argmin(f+P)$ is nonempty and compact. Then we have from \cite[Proposition~2.3]{FBE16} and our assumptions ${\cal X} = \Argmin(f+P)$ and $\gamma \in (0,\frac1L)$ that $\Argmin F_\gamma = \cal X \neq \emptyset$. In particular, it holds that $F_\gamma(x)\ge F_\gamma(\bar x)$ for any $x\in \R^n$ and any $\bar x\in \cal X$.

  Fix an $\bar x\in \cal X$. Then, by assumption and the fact that $\Argmin F_\gamma = \cal X$, there exist $c$ and $\epsilon > 0$ so that \eqref{eq2} holds whenever $\|x-\bar x\|< \epsilon$. Define
  \[
  \Upsilon_\epsilon := \left\{x:\; \|x - \bar x\| + \frac{2}{cc_1c_2}\sqrt{F_\gamma(x) - F_\gamma(\bar x)} < \epsilon\right\}\subseteq U(\bar x,\epsilon),
  \]
  where $c_1$ and $c_2$ are positive numbers so that
  \begin{equation}\label{c1c2}
  \|x - {\cal P}_\gamma (x)\| \ge c_1 \|\nabla F_\gamma(x)\|\ \ {\rm and}\ \ F_\gamma(x) - F_\gamma({\cal P}_\gamma (x))\ge c_2\|x - {\cal P}_\gamma (x)\|^2
  \end{equation}
  for all $x$: the existence of $c_1$ follows from the definition of $\nabla F_\gamma$, while the existence of $c_2$ follows from \cite[Proposition~2.2]{FBE16} and the assumption that $\gamma \in (0,\frac1L)$.

  Consider any $x\in \Upsilon_\epsilon\backslash \cal X$. Define $x^k := {\cal P}_\gamma^k(x)$ for $k\ge 0$ and $k_0 := \inf \{k\ge 0:\; x^k \in {\cal X}\}$. Then it is easy to see that $k_0 \in \{1,2,3,\ldots\}\cup\{\infty\}$ and that $F_\gamma(x^k) > F_\gamma(\bar x)$ whenever $k < k_0$, thanks to the fact that $\Argmin F_\gamma = \cal X$.

  We now establish a key inequality when $x^k\in U(\bar x,\epsilon)$ for some $0\le k < k_0$. In this case,
  \begin{equation}\label{rellong}
    \begin{split}
      &\|x^k - x^{k+1}\| \left[\sqrt{F_\gamma(x^k) - F_\gamma(\bar x)} - \sqrt{F_\gamma(x^{k+1}) - F_\gamma(\bar x)}\right]\\
      & \ge c_1 \|\nabla F_\gamma(x^k)\|\left[\sqrt{F_\gamma(x^k) - F_\gamma(\bar x)} - \sqrt{F_\gamma(x^{k+1}) - F_\gamma(\bar x)}\right]\\
      & \ge c_1 \|\nabla F_\gamma(x^k)\|\left[\frac{1}{2\sqrt{F_\gamma(x^k) - F_\gamma(\bar x)}}(F_\gamma(x^k) - F_\gamma(x^{k+1}))\right]\\
      & \ge \frac{cc_1}{2} (F_\gamma(x^k) - F_\gamma(x^{k+1}))\ge \frac{cc_1c_2}{2} \|x^k - x^{k+1}\|^2,
    \end{split}
  \end{equation}
  where the first inequality follows from \eqref{c1c2} and the fact that $x^{k+1} = {\cal P}_\gamma (x^k)$, the second inequality follows from the concavity of $t\to \sqrt{t}$, the third inequality follows from \eqref{eq2}, and the last inequality follows again from \eqref{c1c2}. Consequently, we have from \eqref{rellong} that if $x^k\in U(\bar x,\epsilon)$ for some $0\le k < k_0$, then
  \begin{equation}\label{rellong2}
    \|x^k - x^{k+1}\| \le \frac{2}{cc_1c_2}\left[\sqrt{F_\gamma(x^k) - F_\gamma(\bar x)} - \sqrt{F_\gamma(x^{k+1}) - F_\gamma(\bar x)}\right].
  \end{equation}

  We next prove by induction that $x^k\in U(\bar x,\epsilon)$ for any $0\le k< k_0$. First of all, it is true that $x^0 = x \in \Upsilon_\epsilon\subseteq U(\bar x,\epsilon)$. Now, suppose that there exists $0\le k < k_0$ so that $x^t\in U(\bar x,\epsilon)$ for all $0\le t\le k$. Then we have
  \begin{equation*}
    \begin{split}
      \|x^{k+1} - \bar x\| & \le \|x^0 - \bar x\| + \sum_{t=0}^{k}\|x^t - x^{t+1}\|\le \|x^0 - \bar x\| + \frac{2}{cc_1c_2}\sqrt{F_\gamma(x^0) - F_\gamma(\bar x)}\\
      & = \|x - \bar x\| + \frac{2}{cc_1c_2}\sqrt{F_\gamma(x) - F_\gamma(\bar x)} < \epsilon,
    \end{split}
  \end{equation*}
  where the second inequality follows from \eqref{rellong2} and the induction hypothesis that $x^t\in U(\bar x,\epsilon)$ for all $0\le t\le k < k_0$, the equality follows from the fact that $x^0=x$, and the last inequality is due to $x\in \Upsilon_\epsilon$. Thus, by induction, we conclude that $x^k\in U(\bar x,\epsilon)$ for any $0\le k< k_0$, from which one can readily deduce that \eqref{rellong2} holds for any $0\le k<k_0$.

  Now, notice that when $k_0 < \infty$, we have $\tilde x := x^{k_0}\in \cal X$ by the definition of $k_0$, and that when $k_0 = \infty$, we have from the convergence theory of the proximal gradient algorithm (see, for example, \cite[Corollary~27.9]{BauCom11}) applied to \eqref{P0} that $\tilde x:= \lim_{k\to\infty} x^k$ exists and $\tilde x\in \cal X$. Thus, in either case, we can sum both sides of \eqref{rellong2} from $k=0$ to $k_0-1$ and deduce that
  \begin{equation*}
  \begin{split}
    {\rm dist}(x,{\cal X})&\le \|x - \tilde x\|\le \sum_{k=0}^{k_0-1}\|x^k - x^{k+1}\| \le \frac{2}{cc_1c_2}\sqrt{F_\gamma(x^0) - F_\gamma(\bar x)} \\
    &\le \frac{2}{c^2c_1c_2}\|\nabla F_\gamma(x)\| \le \frac{2}{c^2c^2_1c_2} \|x - {\cal P}_\gamma (x)\|,
  \end{split}
  \end{equation*}
  where the first inequality follows from the fact that $\tilde x\in \cal X$, the second last inequality follows from \eqref{eq2} and the fact that $x^0 = x$, while the last inequality follows from \eqref{c1c2}.

  Thus, we have shown that, for any $\bar x\in \cal X$, there exist $\epsilon>0$ and $ c_4> 0$ so that
  \begin{equation*}
  {\rm dist}(x,{\cal X})\le c_4 \|x - {\cal P}_\gamma (x)\|
  \end{equation*}
  whenever $x\in \Upsilon_\epsilon$. Combining this with \eqref{ineq0}, we see further that there exist $ c_5> 0$ so that
  \begin{equation}\label{eq3}
  {\rm dist}(x,{\cal X})\le c_5 \|x - {\rm prox}_P(x - \nabla f(x))\|
  \end{equation}
  whenever $x\in \Upsilon_\epsilon$. Since $\cal X$ is compact and $\Upsilon_\epsilon$ is an open set, a standard argument then shows that there exist $\epsilon > 0$ and $c_4 > 0$ so that \eqref{eq3} holds whenever ${\rm dist}(x,{\cal X}) < \epsilon$. In particular, it holds that for any $\xi \ge \inf(f+P)$, there exist $\epsilon > 0$ and $c_5 > 0$ so that \eqref{eq3} holds whenever ${\rm dist}(x,{\cal X}) < \epsilon$ and $f(x) + P(x)\le \xi$. The desired conclusion can now be deduced by using this, the level-boundedness of $f+P$ and following the proof of \cite[Proposition~3]{ZhSo15}. This completes the proof.
\end{proof}

Under condition \eqref{eq} and some standard assumptions, it can be shown that many standard optimization methods applied to minimizing $F_\gamma$ will generate a sequence such that the whole sequence is convergent to a stationary point of $F_\gamma$. As a concrete example, we consider the following proto-typical algorithm (Algorithm~\ref{algorithm}).

\begin{algorithm}[H]
\caption{Algorithm for minimizing $F_\gamma(x)$}
\label{algorithm}
{\bf Parameters} : $0 < \sigma <1$, $0 < \eta <1$, $0 < c_1 < 1 \le c_2 $.
\begin{enumerate}
  \item  Choose an initial point $x^0$. Set $k = 0$.
  \item  Compute a search direction $d^k$ that satisfies the following conditions:
     \begin{align}
     \nabla F_\gamma(x^k)^{T}d^k \le -c_1\|\nabla F_\gamma(x^k)\|\|d^k\|,  \label{angle_cond}\\
     \frac{1}{c_2}\|\nabla F_\gamma(x^k)\| \le \|d^k\| \le c_2\|\nabla F_\gamma(x^k)\|.\label{norm_cond}
     \end{align}
  \item Set $\alpha_k$ to be the largest element in $\{\eta^j:\;j = 0,1,2,...\}$ satisfying
  \begin{equation}\label{armijo}
    F_\gamma(x^k + \alpha_k d^k) \le F_\gamma(x^k) + \sigma\alpha_k\nabla F_\gamma(x^k)^Td^k.
  \end{equation}
  and set $x^{k+1} = x^k + \alpha_k d^k$.
  \item If a termination criterion is not met, update $k \leftarrow k + 1$ and go to Step 2.
\end{enumerate}

\end{algorithm}

In this algorithm, in each iteration, one computes a search direction that is gradient-related in the sense that it satisfies \eqref{angle_cond} and \eqref{norm_cond}. Note that these two conditions are trivially satisfied by the choice of $d^k = - \nabla F_\gamma(x^k)$, for example. One then perform a line-search via backtracking to satisfy the Armijo condition \eqref{armijo}. Since $F_\gamma$ is continuously differentiable, and at any nonstationary point $x^k$, it holds that
\[
\nabla F_\gamma(x^k)^Td^k \le -c_1\|\nabla F_\gamma(x^k)\|\|d^k\| \le -\frac{c_1}{c_2}\|\nabla F_\gamma(x^k)\|^2< 0
\]
due to \eqref{angle_cond} and \eqref{norm_cond}, one can readily show that for any such $x^k$ the Armijo condition \eqref{armijo} must be satisfied for some sufficiently small $\alpha_k > 0$. On the other hand, it is clear from \eqref{norm_cond} that if $x^k$ is a stationary point, then $d^k = 0$ and inductively, $x^{k+l} = x^k$ and $\alpha_{k+l} = 1$ for all $l\ge 0$. Combining the above observations, we conclude that the sequences $\{\alpha_k\}$ and $\{x^k\}$ are well defined. We next show that the stepsize sequence $\{\alpha_k\}$ is indeed uniformly bounded away from zero if $f+P$ is coercive and $f$ is, in addition, analytic.
The proof is standard; see, for example, \cite{TseY09}. We include a simple proof for completeness.

\begin{lemma}\label{lem2}
Suppose that $f+P$ is coercive, $f$ is analytic and $\gamma \in (0,\frac1L)$. Then there exists $\underline{\alpha} > 0$ such that $\alpha_k \in [\underline{\alpha},1]$ for all $k$, where $\{\alpha_k\}$ is generated by Algorithm \ref{algorithm}.
\end{lemma}
\begin{proof}
Note from Theorem~\ref{thm1} that $F_\gamma$ is level bounded. This together with \eqref{armijo} and \eqref{angle_cond} shows that $\{x^k\}\subseteq \{x:\; F_\gamma(x)\le F_\gamma(x^0)\}\subseteq B(0,R_1)$ for some $R_1 > 0$. Additionally, we have from \eqref{norm_cond} and the definition of $\nabla F_\gamma$ that
\[
\begin{split}
\|d^k\| &\le c_2 \|\nabla F_\gamma(x^k)\| = c_2\|\gamma^{-1}(I - \gamma \nabla^2 f(x^k))(x^k - {\rm prox}_{\gamma P}(x^k - \gamma \nabla f(x^k)))\|\\
        &\le \frac{c_2}{\gamma}(1 + \gamma L)\|x^k - {\rm prox}_{\gamma P}(x^k - \gamma \nabla f(x^k))\| \le R_2
\end{split}
\]
for some $R_2 > 0$, where the second inequality follows from the definition of $L$, while the last inequality follows from the continuity of
$x\mapsto \|x - {\rm prox}_{\gamma P}(x - \gamma \nabla f(x))\|$ and the fact that $\{x^k\}\subseteq B(0,R_1)$. Consequently, it holds that
\begin{equation}\label{compactcontain}
\{x^k + \alpha d^k:\; \alpha\in [0,1], k = 0,1,\ldots\}\subseteq B(0,R),
\end{equation}
where $R = R_1 + R_2$.

Next, due to the analyticity of $f$, we see that $x\mapsto \nabla F_\gamma(x)$ is Lipschitz continuous on the compact set $B(0,R)$. Denote a Lipschitz constant of $\nabla F_\gamma$ on $B(0,R)$ by $L_F$. Then, because of \eqref{compactcontain}, we have for any $\alpha\in [0,1]$ and any $k\ge 0$ that
\begin{equation*}
    F_\gamma(x^k+\alpha d^k) \le F_\gamma(x^k) + \alpha\nabla F_\gamma(x^k)^Td^k + \frac{L_F\alpha^2}{2}\|d^k\|^2.
\end{equation*}
Rearranging terms in the above relation and invoking \eqref{angle_cond} and \eqref{norm_cond}, we obtain further that
\begin{equation*}
\begin{split}
  F_\gamma(x^k+\alpha d^k) - F_\gamma(x^k) & \le  (1 - \sigma)\alpha\nabla F_\gamma(x^k)^Td^k + \frac{L_F\alpha^2}{2}\|d^k\|^2 + \sigma\alpha\nabla F_\gamma(x^k)^Td^k \\
                             & \le  -c_1(1 - \sigma)\alpha\|\nabla F_\gamma(x^k)\|\|d^k\| + \frac{L_F\alpha^2}{2}\|d^k\|^2 + \sigma\alpha\nabla F_\gamma(x^k)^Td^k\\
                             & \le  -\frac{c_1(1 - \sigma)\alpha}{c_2}\|d^k\|^2 + \frac{L_F\alpha^2}{2}\|d^k\|^2 + \sigma\alpha\nabla F_\gamma(x^k)^Td^k \\
                             & =  \frac{L_F\alpha}{2}\|d^k\|^2\left(\alpha - \frac{2c_1(1-\sigma)}{c_2L_F}\right) + \sigma\alpha\nabla F_\gamma(x^k)^Td^k,
\end{split}
\end{equation*}
where the second inequality follows from \eqref{angle_cond}, while the third inequality follows from \eqref{norm_cond}.
Therefore, the Armijo condition \eqref{armijo} holds whenever $\alpha \le \frac{2c_1(1-\sigma)}{c_2L_F}$. Since $\alpha_k \in \{\eta^j:\; j = 0,1,2,...\}$, we must then have $1\ge \alpha_k \ge\underline{\alpha}:= \min\{1,\frac{2c_1(1-\sigma)\eta}{c_2L_F}\}$ for all $k$. This completes the proof.
\end{proof}

We now show that the whole sequence generated by Algorithm~\ref{algorithm} is convergent under suitable assumptions, and establish a local linear convergence rate of the sequence when Assumption~\ref{assum1} is satisfied. The proof technique is standard in the literature; similar kinds of results can be found in \cite{NR13}, which studied a slightly different algorithm. We include the details for the ease of the readers.

\begin{proposition}{\bf (Convergence of Algorithm~\ref{algorithm})}\label{convergence}
Suppose that $f+P$ is coercive, the function $f$ is analytic, $P$ is continuous on ${\rm dom}\,\partial P$, subanalytic and bounded below (i.e., $\inf P > -\infty$), and $\gamma\in(0,\frac{1}{L})$. Let $\{x^k\}$ be the sequence generated by Algorithm \ref{algorithm}. Then the following statements hold true.
\begin{enumerate}[{\rm (a)}]
    \item The whole sequence $\{x^k\}$ converges to a stationary point $\bar{x}$ of $F_\gamma$.
    \item Suppose in addition that Assumption~\ref{assum1} holds. Then the sequences $\{x^k\}$ and $\{F_\gamma(x^k)\}$ are locally linearly convergent.
  \end{enumerate}
\end{proposition}

\begin{proof}
(a) If $\nabla F_\gamma(x^k) = 0$ for some $k$, then $x^k$ is a stationary point of $F_\gamma$. Moreover, since our choice of $d^k$ satisfies \eqref{norm_cond}, we must have $x^{k+l} = x^k$ for all $l \ge 0$. Thus, to prove part (a), it suffices to consider the case when $\nabla F_\gamma(x^k) \neq 0$ for all $k$. In this case, we have from \eqref{armijo} that
\begin{align}
  F_\gamma(x^{k+1}) - F_\gamma(x^k) & \le  \sigma\alpha_k\nabla F_\gamma(x^k)^Td^k \stackrel{\rm (i)}{\le}  -c_1\sigma\underline{\alpha}\|\nabla F_\gamma(x^k)\|\|d^k\| \nonumber\\
                                    & \stackrel{\rm (ii)}{\le}   -\frac{c_1\sigma\underline{\alpha}}{c_2}\|\nabla F_\gamma(x^k)\|^2 \stackrel{\rm (iii)}{\le}  -\frac{c_1\sigma\underline{\alpha}}{c_2^3}\|d^k\|^2 \label{Fmono_F}\\
                                    & \stackrel{\rm (iv)}{\le}  -\frac{c_1\sigma\underline{\alpha}}{c_2^3}\|x^{k+1} - x^k\|^2,\label{Fmono_x}
\end{align}
where (i) follows from \eqref{angle_cond} and Lemma~\ref{lem2}, (ii) and (iii) follow from \eqref{norm_cond}, while (iv) follows from the definition of $x^{k+1}$.
In particular, we see that the sequence $\{F_\gamma(x^k)\}$ is nonincreasing. Since $F_\gamma$ is also level bounded due to Theorem~\ref{thm1} and the coerciveness of $f + P$, we conclude further that $\{x^k\}$ is bounded. Consequently, the nonincreasing sequence $\{F_\gamma(x^k)\}$ is bounded from below and is therefore convergent. Summing both sides of \eqref{Fmono_F} from $0$ to $\infty$ gives
\[
-\infty < \lim_{k\to \infty}F_\gamma(x^{k+1}) - F_\gamma(x^0) \le -\frac{c_1\sigma\underline{\alpha}}{c_2}\sum_{k=0}^\infty\|\nabla F_\gamma(x^k)\|^2,
\]
showing that $\lim_{k\rightarrow\infty}\nabla F_\gamma(x^k)=0$. Then it is routine to show from this and the continuity of $\nabla F_\gamma$ that any accumulation point of $\{x^k\}$, which exists due to the boundedness of $\{x^k\}$, is a stationary point of $F_\gamma$. To complete the proof, it remains to show that the whole sequence $\{x^k\}$ is indeed convergent.

To proceed, for notational simplicity, we assume without loss of generality that $\lim_{k\to\infty}F_\gamma(x^k)=0$. If $F_\gamma(x^k) = 0$ for some $k\ge 0$, since $F_\gamma$ is nonincreasing, we must then have $F_\gamma(x^{k+l}) = 0$ for any $l\ge 0$. Consequently, we see from \eqref{Fmono_F} that $\nabla F_\gamma(x^{k+l}) = 0$ for any $l \ge 0$, a contradiction to our assumption that $\nabla F_\gamma(x^t) \neq 0$ for any $t$. Thus, it remains to consider the case where $F_\gamma(x^k) > 0$ for all $k$.

In this case, let $\Omega$ denote the set of accumulation points of $\{x^k\}$, which is clearly a compact set and satisfies ${\rm dist}(x^k, \Omega)\to 0$. In addition, it is routine to show that $F_\gamma(\bar x) = 0$ whenever $\bar x \in \Omega$. Using these, Theorem \ref{thm2} (a) and \cite[Lemma 1]{AB09}, we conclude that there exist an exponent $\theta\in (0,1)$, $c >0$ and $N_0 > 0$ such that
\begin{equation}\label{KL_ineq}
    cF_\gamma(x^k)^{\theta} = c|F_\gamma(x^k)|^{\theta} \le \|\nabla F_\gamma(x^k)\|
\end{equation}
for all $k \ge N_0$. Combining this with the concavity of the function $s\to s^{1-\theta}$ (for $s > 0$), we have
\begin{equation*}
\begin{split}
    F_\gamma(x^k)^{1-\theta} - F_\gamma(x^{k+1})^{1-\theta} &\ge (1-\theta)F_\gamma(x^{k})^{-\theta}(F_\gamma(x^k) - F_\gamma(x^{k+1}))\\
    & \ge c(1-\theta)\frac{F_\gamma(x^{k}) - F_\gamma(x^{k+1})}{\|\nabla F_\gamma(x^k)\|}\\
    & \ge C_1 \|\nabla F_\gamma(x^k)\|\ge C_2 \|x^{k+1}-x^k\|,
\end{split}
\end{equation*}
for some positive constants $C_1$ and $C_2$, where the second inequality follows from \eqref{KL_ineq}, the third inequality follows from \eqref{Fmono_F} and the last inequality follows from \eqref{Fmono_x}.
Summing both sides of the above relation from $N_0$ to $\infty$, we see further that
\[
    \sum_{k=N_0}^\infty\|x^{k+1} - x^k\|  \le  \frac{1}{C_2}[F_\gamma(x^{N_0})^{1-\theta} - \lim_{N\to \infty}F_\gamma(x^{N+1})^{1-\theta}]
    \le   \frac{1}{C_2}F_\gamma(x^{N_0})^{1-\theta},
\]
which implies that $\{x^k\}$ is a Cauchy sequence. Thus, the whole sequence $\{x^k\}$ is actually convergent. This completes the proof.

(b) Suppose in addition that Assumption~\ref{assum1} holds and we again focus on the case where $\nabla F_\gamma(x^k)\neq 0$ for all $k$. Then in view of \eqref{eq1}, the convergence of $\{x^k\}$ to a stationary point $\bar x$ of $F_\gamma$ and the fact that $\{F_\gamma(x^k)\}$ is nonincreasing, we conclude that there exist $c > 0$ and an integer $N$ such that
\begin{equation*}
    \|\nabla F_\gamma(x^k)\| \ge c\sqrt{F_\gamma(x^k) - F_\gamma(\bar{x})}
\end{equation*}
whenever $k \ge N$. Combining this with \eqref{Fmono_F}, we obtain further that
\begin{equation*}
\begin{split}
  F_\gamma(x^k) - F_\gamma(x^{k+1}) & \ge \frac{c_1\sigma\underline{\alpha}}{c_2}\|\nabla F_\gamma(x^k)\|^2 \ge C_0(F_\gamma(x^k) - F_\gamma(\bar{x})),
\end{split}
\end{equation*}
for some $C_0 > 0$, and one can always choose $C_0 \in (0,1)$ without loss of generality, since $F_\gamma(x^k)\ge F_\gamma(\bar x)$ for all $k$.
Then we have, upon rearranging terms, that
\begin{equation*}
   F_\gamma(x^{k+1}) - F_\gamma(\bar x) \le (1- C_0) [ F_\gamma(x^k) - F_\gamma(\bar x)]
\end{equation*}
whenever $k\ge N$. Since $1- C_0\in (0,1)$, this shows that $\{F_\gamma(x^k)\}$ is (at least) $Q$-linearly convergent.

Next, we note from \eqref{Fmono_x} that there exist $C$, $C_3 > 0$ and $\eta\in (0,1)$ such that
\[
\|x^{k+1} - x^k\| \le C\sqrt{F_\gamma(x^k) - F_\gamma(x^{k+1})} \le C\sqrt{F_\gamma(x^k) - F_\gamma(\bar x)} \le C_3 \eta^k,
\]
where the second inequality follows from the fact that $\{F_\gamma(x^k)\}$ is nonincreasing, while the last inequality follows from the fact that $\{F_\gamma(x^k)\}$ is $Q$-linearly convergent.
Consequently, we have
\[
\|x^k - \bar x\| \le \sum_{t = k}^\infty\|x^{t+1} - x^t\| \le \frac{C_3\eta^k}{1-\eta},
\]
showing that $\{x^k\}$ is (at least) $R$-linearly convergent.
This completes the proof.
\end{proof}

On passing, we would like to point out that the algorithm we consider here is different from those considered in \cite{FBE16}. In \cite{FBE16}, in each iteration of their algorithms, after moving along the search direction $d^k$ with a certain stepsize, they induce sufficient descent by performing one step of proximal gradient algorithm on the function $f+P$. In contrast, we choose $d^k$ to satisfy \eqref{angle_cond} and \eqref{norm_cond} without resorting to the proximal gradient update in each iteration.

\section{Applications to difference-of-convex programming}\label{sec:DC}

In this section, we describe a class of problems that can be reformulated into \eqref{P0}. We also derive sufficient conditions so that the results in Section~\ref{sec:FBE} can be applied to deducing the level-boundedness of the corresponding forward-backward envelope.

We first describe our class of problems. This is a class of regularized least squares problems and arises frequently in applications such as statistical machine learning. The problems take the following form:
\begin{equation}\label{P1}
  v_{\mu_1,\mu_2}:= \min_z \ \ J(z) := \frac12 \|Az - b\|^2 + \mu_1 H_1(z) - \mu_2 H_2(z),
\end{equation}
where $\mu_1\ge \mu_2 > 0$, $A\in \R^{m\times n}$ and $b\in \R^m$, the regularization functions $H_1(z)$ and $H_2(z)$ are proper closed convex functions and there exists a norm $\rho$ such that
\[
0\le H_2(z)\le \min\{\rho(z),H_1(z)\}\ \ \ \mbox{for all}\ z.
\]
In particular, $H_2$ is continuous, and we also conclude from the above assumptions that $v_{\mu_1,\mu_2} \ge 0$ for any $\mu_1\ge \mu_2 > 0$.
Concrete examples of regularization functions that satisfy the above assumptions include
\begin{itemize}
  \item $H_1(z) = \|z\|_1$ and $H_2(z) = \|z\|$, and $\mu_1 = \mu_2$. This is known as the $\ell_{1-2}$ regularization; see, for example, \cite{YLHX14};
  \item $H_1(z) = \|z\|_1$ and $H_2(z) = \sum_{i=1}^k|z_{[i]}|$, where $z_{[i]}$ denotes the $i$th largest element in magnitude, and $k\le n$; see, for example, \cite{WLZ15}.
  \item $H_1(z) = \|z\|_1$ and $H_2(z) = \sum_{i=1}^n\int_0^{|z_i|}\frac{[\min(\theta\lambda,t)-\lambda]_+}{(\theta-1)\lambda}dt~(\theta>2)$, and $\mu_1 = \mu_2 = \lambda$. This is known as the smoothly clipped absolute deviation (SCAD) regularization; see, for example, \cite{FL11,Gong13}.
  \item $H_1(z) = \|z\|_1$ and $H_2(z) = \sum_{i=1}^n\int_0^{|z_i|}\min(1,t/(\theta\lambda))dt~(\theta > 0)$, and $\mu_1 = \mu_2 = \lambda$. This is known as the minimax concave penalty (MCP) regularization; see, for example, \cite{Gong13,Zhang10}.
\end{itemize}

Since $H_2$ can be nonsmooth in general as in the first two examples above, it appears that \eqref{P1} does not readily take the form of \eqref{P0}. Nonetheless, one can equivalently reformulate \eqref{P1} as
\begin{equation}\label{P2}
  \min_{y,z} \ \  \underbrace{\frac12 \|Az - b\|^2 - \mu_2\langle y,z\rangle}_{f} + \underbrace{\mu_1 H_1(z) + \mu_2H_2^*(y)}_{P},
\end{equation}
where $H_2^*(y) := \sup_{z}\{\langle y,z\rangle - H_2(z)\}$ is the convex conjugate of $H_2$. This is in the form of \eqref{P0}. Thus, one can then consider minimizing the corresponding forward-backward envelope $F_\gamma$ instead. In the next proposition, we give a sufficient condition so that the $F_\gamma$ corresponding to \eqref{P2} is level bounded.

\begin{proposition}\label{prop4.1}{\bf (Coerciveness of $f + P$ in \eqref{P2})}
  Suppose that the objective function $J$ in \eqref{P1} is coercive.
  Then the objective function in \eqref{P2} is coercive. Consequently, for any $\gamma\in (0,\frac1L)$, the corresponding $F_\gamma$ is level bounded.
\end{proposition}
\begin{proof}
  Since $H_2(z) \le \rho(z)$ for all $z$, we have from the definition of conjugate functions that $H_2^*(u) \ge \rho^*(u) = \delta_{\rho^\circ(\cdot)\le 1}(u)$ for all $u$, where $\rho^\circ$ is the dual norm of $\rho$. Using this and the definitions of $f$ and $P$ in \eqref{P2}, we see that
  \begin{equation}\label{ineq}
  \begin{split}
    (f+P)(y,z) & = \frac12 \|Az - b\|^2 - \mu_2\langle y,z\rangle + \mu_1 H_1(z) + \mu_2H_2^*(y)\\
    & = \frac12 \|Az - b\|^2 - \mu_2\langle y,z\rangle + \mu_1 H_1(z) + \mu_2H_2^*(y) + \delta_{\rho^\circ(\cdot)\le 1}(y)\\
    & \ge \frac12 \|Az - b\|^2 + \mu_1 H_1(z) - \mu_2H_2(z) + \delta_{\rho^\circ(\cdot)\le 1}(y) = J(z) + \delta_{\rho^\circ(\cdot)\le 1}(y).
  \end{split}
  \end{equation}
   From our assumption, $J$ is coercive. In addition, the function $y\mapsto \delta_{\rho^\circ(\cdot)\le 1}(y)$ is level bounded, and hence coercive due to \cite[Proposition~3.1.3]{AuT03}. Then it is routine to check that the function $(y,z)\mapsto J(z) + \delta_{\rho^\circ(\cdot)\le 1}(y)$ is also coercive. Consequently, $f + P$ is coercive in view of \eqref{ineq}. As a consequence of this and Theorem~\ref{thm1}, we conclude that, for any $\gamma\in (0,\frac1L)$, the corresponding $F_\gamma$ is level bounded. This completes the proof.
\end{proof}

We now present some sufficient conditions for the function $J$ in \eqref{P1} to be coercive.

\begin{proposition}\label{corol}{\bf (Coerciveness of $J$)}
The function $J$ in \eqref{P1} is coercive when
\begin{enumerate}[{\rm (a)}]
  \item $\mu_1 > \mu_2$ and $H_1$ is level bounded; or
  \item $\mu_1 = \mu_2$ and $H_1$ and $H_2$ are norms such that $H_1(z) > H_2(z)$ whenever $\|z\|_0 \ge r_A+1$, where $r_A = \max\{i:\;\mbox{Any $i$ columns of $A$ are linearly independent.}\}$.
\end{enumerate}
\end{proposition}
\begin{proof}
(a) Recall from \eqref{P1} and the definition of $v_{\mu_2,\mu_2}$ that
  \begin{equation*}
  \begin{split}
    J(z) = \frac12 \|Az - b\|^2 + \mu_1 H_1(z) - \mu_2H_2(z)
    \ge v_{\mu_2,\mu_2} + (\mu_1-\mu_2)H_1(z).
  \end{split}
  \end{equation*}
  Since $H_1$ is proper closed convex and level bounded, it is also coercive according to \cite[Proposition~3.1.3]{AuT03}. The coerciveness of $J$ now follows immediately.

  (b) Suppose to the contrary that $J$ is not coercive. Since $H_2$ is a norm, we then have from the definition that
      \begin{equation*}
         \liminf_{H_2(z)\rightarrow\infty}\frac{J(z)}{H_2(z)} \le 0.
      \end{equation*}
Consequently, there exists a sequence $\{z^k\}$ such that $H_2(z^k)\rightarrow\infty$ and
\begin{equation*}
    \frac{\frac{1}{2}\|Az^k - b\|^2 + \mu(H_1(z^k) - H_2(z^k))}{H_2(z^k)} \le \frac{1}{k};
\end{equation*}
here, we use $\mu$ to denote the common value of $\mu_1$ and $\mu_2$.
Since $H_1(z^k) \ge H_2(z^k)$, we have further that for all $k$,
\begin{equation}\label{rel0}
  \frac{1}{2}\|Az^k - b\|^2   \le  \frac{H_2(z^k)}{k} \ \ \ {\rm and}\ \ \
  \mu(H_1(z^k) - H_2(z^k))  \le   \frac{H_2(z^k)}{k}.
\end{equation}
Since $H_2$ is a norm, by passing to a subsequence if necessary, we assume without loss of generality that $\frac{z^k}{H_2(z^k)}\rightarrow d$ for some $d\in \R^n$. Thus, $H_2(d) = 1$ and hence $d\neq 0$ in particular. Dividing both sides of the first and second inequalities in \eqref{rel0} by $[H_2(z^k)]^2$ and $H_2(z^k)$ respectively, using the fact that $H_1$ is a norm and passing to the limit, we obtain further that
\begin{equation}\label{rel}
       Ad = 0\ \ \ {\rm and}\ \ \ H_1(d)\le 1.
\end{equation}
The first relation in \eqref{rel} together with the fact that $d\neq 0$ and the assumption on the columns of $A$ implies that $\|d\|_0 \ge r_A + 1$.
Hence, we have from the assumption that $H_2(d) < H_1(d)$. This together with the second relation in \eqref{rel}
gives $H_2(d) < H_1(d) \le 1$, which is a contradiction to the fact that $H_2(d) = 1$. This completes the proof.
\end{proof}

\begin{example}\label{example}
We give some concrete examples satisfying the conditions in Proposition~\ref{corol}:
\begin{enumerate}[{\rm (a)}]
  \item $\mu_1 > \mu_2$ and $H_1(z) = \|z\|_p$, $p\in [1,\infty)$: such $H_1$ are clearly level bounded.
  \item $\mu_1 = \mu_2$ and $H_1(z) = \|z\|_1$, with $H_2(z) = \|z\|_p$, $p \in (1,\infty)$, and $A$ does not have zero columns. In this case, it is easy to see that $r_A\ge 1$. Moreover, $\|z\|_1 > \|z\|_p$ whenever $\|z\|_0 > 1$ and $p \in (1,\infty)$. This is a consequence of the fact that
  $z_1^p + \cdots + z_n^p < 1$ for any $z\ge 0$ satisfying $\sum_{i=1}^n z_i = 1$ and $\|z\|_0 > 1$. In particular, $\|z\|_1 > \|z\|_p$ whenever $\|z\|_0 \ge r_A+1(>1)$ and $p \in (1,\infty)$.
  \item $\mu_1=\mu_2$ and $H_1(z) = \|z\|_1$, with $H_2(z) = \sum_{i=1}^k|z_{[i]}|$, where $z_{[i]}$ denotes the $i$th largest element in magnitude, $k\le {\rm rank}(A)$, and any $k$ columns of $A$ are linearly independent. In this case, we have $r_A\ge k$, and moreover $\|z\|_1>\sum_{i=1}^k|z_{[i]}|$ whenever $\|z\|_0\ge k+1$ and hence, whenever $\|z\|_0 \ge r_A+1$.
\end{enumerate}
\end{example}

Before ending this section, we derive a bound $L$ on the magnitude of the eigenvalues of $\nabla^2 f(y,z)$ for the $f$ in \eqref{P2}, which is necessary for obtaining an upper bound on the $\gamma$ used in the corresponding forward-backward envelope. Note that for any $(y,z)$, we have
\begin{equation*}
  \nabla^2 f(y,z) = \begin{bmatrix}
    0 & -\mu_2I \\ -\mu_2I & A^TA
  \end{bmatrix}.
\end{equation*}
Hence, the operator norm can be upper bounded as follows
\begin{equation*}
\begin{split}
  \|\nabla^2 f(y,z)\| & = \sup_{\|(u,v)\| = 1}\left\|\begin{bmatrix}
    0 & -\mu_2I \\ -\mu_2I & A^TA
  \end{bmatrix} \begin{bmatrix}
    u\\v
  \end{bmatrix}\right\| \le \sup_{\|(u,v)\| = 1}\left\|\begin{bmatrix}
    0 & \mu_2 \\ \mu_2 & \|A^TA\|
  \end{bmatrix} \begin{bmatrix}
    \|u\|\\ \|v\|
  \end{bmatrix}\right\|\\
  & \le \lambda_{\max}\left(\begin{bmatrix}
    0 & \mu_2 \\ \mu_2 & \|A^TA\|
  \end{bmatrix}\right) = \frac{\lambda_{\max}(A^TA) + \sqrt{\lambda^2_{\max}(A^TA) + 4\mu_2^2}}{2},
\end{split}
\end{equation*}
where $\|A^TA\|$ denotes the operator norm of $A^TA$, which is equal to $\lambda_{\max}(A^TA)$, the maximum eigenvalue of $A^TA$.
Thus, one can set $L=\frac{\lambda_{\max}(A^TA) + \sqrt{\lambda^2_{\max}(A^TA) + 4\mu_2^2}}{2}$.

\section{Numerical experiments on least squares problems with $\ell_{1-2}$ regularization}\label{sec:num}

In this section, we perform numerical experiments to test the efficiency of solving \eqref{P1} via minimizing the corresponding $F_\gamma$. All the experiments are performed in MATLAB version R2015b on a desktop computer with a 3.6GHz CPU and 32G RAM, and all codes are written in MATLAB.

In our experiments, we take the least squares problems with $\ell_{1-2}$ regularization as our test problems. This class of problem is given by
\begin{equation}\label{P-1}
  \min_z \ h(z) := \frac12 \|Az - b\|^2 + \mu(\|z\|_1 - \|z\|),
\end{equation}
where $A\in \R^{m\times n}$ does not have zero columns, $b\in \R^m$ and $\mu > 0$.
This model is a special case of \eqref{P1} with $\mu_1 = \mu_2 = \mu$, $H_1(z) = \|z\|_1$ and $H_2(z) = \|z\|$, and has been considered in \cite{YLHX14} for sparse recovery.

We compare three different approaches for solving \eqref{P-1}: {\bf FBE$_{\rm L-BFGS}$}, {\bf NPG} and {\bf NPG$_{\rm major}$}. The first approach is based on our discussion of the forward-backward envelope, while the other two approaches are standard applications of proximal gradient type algorithms for solving \eqref{P-1}. These two latter approaches are included here as benchmark. We now discuss these approaches in further details below.

\paragraph{FBE$_{\rm L-BFGS}$.} In this approach, we apply Algorithm~\ref{algorithm} to minimizing $F_\gamma$ with $\gamma = 0.95/L$ under a specific choice of $d^k$ to be made explicit below. Here, $F_\gamma$ is the forward-backward envelope corresponding to the following equivalent reformulation of \eqref{P-1}:
\begin{equation}\label{P4}
\min_{y,z} \ \underbrace{\frac12 \|Az - b\|^2 - \mu\langle y,z\rangle}_{f} + \underbrace{\mu\|z\|_1 + \delta_{B(0,1)}(y)}_{P}
\end{equation}
and $L$ is computed as in the end of Section~\ref{sec:DC}.\footnote{$\lambda_{\max}(A^TA)$ is computed via the MATLAB code {\sf opts.issym = 1;
lambda= eigs(A*A',1,'LM',opts);} when $m > 2000$, and by {\sf lambda = norm(A*A')} otherwise.} It is easy to see that $f$ is analytic, $P$ is continuous on its domain, bounded below and subanalytic. Moreover, since $A$ has no zero columns, $f + P$ is coercive in view of Example~\ref{example} (b), Proposition~\ref{corol} (b) and Proposition~\ref{prop4.1}. Thus, according to Proposition~\ref{convergence}, the whole sequence generated by Algorithm~\ref{algorithm} converges to a stationary point of \eqref{P4}.

In our experiments below, we set $\sigma = 10^{-4}$, $\eta = 0.5$, $c_1 = \frac{1}{c_2} = 10^{-5}$ in Algorithm~\ref{algorithm}. Moreover, in each iteration, we compute a $d_{\rm B}$ as the output of \cite[Algorithm~9.1]{NoWr99} using a memory of $10$, and set
\[
d^k = \begin{cases}
  -d_{\rm B} & \mbox{if $-d_{\rm B}$ satisfies \eqref{angle_cond} and \eqref{norm_cond},}\\
  -\nabla F_\gamma(x^k) & \mbox{otherwise};
\end{cases}
\]
i.e., we use the limited-memory BFGS search direction, and resort to the steepest descent direction if \eqref{angle_cond} or \eqref{norm_cond} fails.\footnote{We note on passing that the computation of $\nabla F_\gamma$ is simple: it involves the proximal mapping of $P$ in \eqref{P4}, which boils down to evaluating the $\ell_1$ shrinkage operator (proximal mapping of $\ell_1$ norm) and the projection onto $B(0,1)$. In addition, in our numerical experiments below, the steepest descent direction was {\em never} invoked.}
Finally, we initialize the algorithm at the origin, and terminate it when
\[
\frac{\|\nabla F_\gamma(x^{k})\|}{\max\{1,F_\gamma(x^{k})\}} < tol
\]
for some $tol > 0$.

\paragraph{NPG.} In this approach, we apply the nonmonotone proximal gradient method discussed in \cite{WrNoFi09} (see also \cite{Gong13} and \cite[Appendix~A]{ChLuPo14}) for solving \eqref{P-1}. Following the notation in \cite[Appendix~A, Algorithm~1]{ChLuPo14}, in our experiments, we apply the method with $f(z) = \frac12\|Az - b\|^2$ and $P(z) = \mu(\|z\|_1 - \|z\|)$, and set $\tau = 2$, $c = 10^{-4}$, $M = 4$, $L^0_0=1$ and
\[
L_k^0:= \min\left\{\max\left\{\frac{\|A(z^k - z^{k-1})\|^2}{\|z^k - z^{k-1}\|^2},10^{-8}\right\},10^8\right\}
\]
for $k\ge 1$. We note that the subproblem in \cite[Equation~A.5]{ChLuPo14} now becomes
\begin{equation}\label{subproblem2}
\min_z \ \ \langle A^T(Az^k - b),z - z^k\rangle + \frac{L_k}{2}\|z - z^k\|^2 + \mu(\|z\|_1 - \|z\|);
\end{equation}
we will discuss its closed form solution in Appendix~\ref{AppendixA}. Finally, we initialize the algorithm at the origin, and terminate it when
\[
\frac{\|z^k - z^{k-1}\|}{\max\{1,h(z^k)\}} < tol,
\]
where $tol > 0$.

\paragraph{NPG$_{\rm major}$.} This approach is the same as {\bf NPG} except that in each iteration, the subproblem takes the following form (in place of \eqref{subproblem2}):
\begin{equation}\label{subproblem3}
\min_z \ \ \langle A^T(Az^k - b) - \mu \xi^k,z - z^k\rangle + \frac{L_k}{2}\|z - z^k\|^2 + \mu \|z\|_1,
\end{equation}
for a fixed $\xi^k\in \partial \|z^k\|$; i.e., we replace the function $-\|z\|$ by its majorant $-\|z^k\| - \langle \xi^k,z - z^k\rangle$. Note that the subproblem \eqref{subproblem3} has a closed form solution in terms of the $\ell_1$ shrinkage operator.
In our experiments, we use the same parameters $\tau$, $c$, $M$ and $L_k^0$ as in {\bf NPG}. We initialize this algorithm also at the origin, and terminate it when
\[
\frac{\|z^k - z^{k-1}\|}{\max\{1,h(z^k)\}} < tol,
\]
where $tol > 0$.

In our first experiment, we compare {\bf FBE$_{\rm L-BFGS}$} against {\bf NPG} and {\bf NPG$_{\rm major}$} for solving \eqref{P-1} on randomly generated instances.  %
%
%
%
These random instances are generated as follows.
We first generate a matrix $A\in \R^{m\times n}$ with i.i.d. standard Gaussian entries.\footnote{Thus, with high probability, $A$ does not have zero columns.} The matrix is further normalized so that each column has unit norm. Next, we choose an index set $T\subseteq \{1,\ldots,n\}$ of size $s$ uniformly at random and generate a vector $y\in \R^{s}$ with i.i.d. standard Gaussian entries. The measurement vector $b$ is then generated as $b = A_T y + \sigma \hat n$, where $A_T$ is the submatrix formed by those columns of $A$ indexed by $T$, $\sigma > 0$ and $\hat n\in \R^{m}$ is a random vector with i.i.d. standard Gaussian entries.

In our numerical tests, for each $(m,n,s) = (720i,2560i,160i)$ for $i=1,\ldots,10$, we generate $10$ random instances as described above with $\sigma = 10^{-2}$.\footnote{The dimension parameters are similar to those used in \cite[Section~3]{LuPZh12}, except that our $s$ is twice as large. This indicates that the test instances we consider here are harder in terms of sparse recovery.} We terminate {\bf FBE$_{\rm L-BFGS}$} at $tol = 10^{-6}$, and the other two algorithms at $tol = 10^{-4}$. The computational results are presented in Tables~\ref{t1} and \ref{t2}, which correspond to \eqref{P-1} with $\mu = 5\times 10^{-4}$ and $10^{-3}$, respectively. In these tables, we denote the algorithm {\bf FBE$_{\rm L-BFGS}$} by FBE, the algorithm {\bf NPG} by NPG and the algorithm {\bf NPG$_{\rm major}$} by Major. We report the time for computing $\lambda_{\max}(A^TA)$ ({\sc t}$_{\lambda_{\max}}$), the number of iterations (iter), CPU time in seconds (CPU)\footnote{The CPU times under the FBE column do {\em not} include the times for computing $\lambda_{\max}(A^TA)$. The latter are reported separately in the fourth column of each table.} and the terminating function values (fval)\footnote{For all three algorithms, we output the $z^k$ at termination and compute the function value $h(z^k)$.}, averaged over the $10$ random instances. We see that {\bf FBE$_{\rm L-BFGS}$} generally outperforms the other two algorithms in terms of both the number of iterations and CPU time. Moreover, the function values obtained at termination from all three algorithms are comparable, with {\bf FBE$_{\rm L-BFGS}$} giving slightly smaller function values. In addition, the performance of {\bf FBE$_{\rm L-BFGS}$} becomes better as the dimension increases and $\mu$ decreases.

\begin{table}[h]
\caption{\small Results for random $A$ with unit column norms, $\mu = 5\times 10^{-4}$}
 \label{t1}
\hspace{-0.8 cm}
\begin{footnotesize}
\begin{tabular}{|c c c ||r||r r r||r r r||r r r|}
\hline \multicolumn{3}{|c||}{size} & & \multicolumn{3}{c||}{iter} &
\multicolumn{3}{c||}{CPU} & \multicolumn{3}{c|}{fval}\\

\multicolumn{1}{|c}{ $m$} & \multicolumn{1}{c}{ $n$} &
\multicolumn{1}{c||}{ $s$} & {\sc t$_{\lambda_{\max}}$} &
\multicolumn{1}{c}{\sc FBE} & \multicolumn{1}{c}{\sc NPG} & \multicolumn{1}{c||}{\sc Major} &
\multicolumn{1}{c}{\sc FBE} & \multicolumn{1}{c}{\sc NPG} & \multicolumn{1}{c||}{\sc Major} &
\multicolumn{1}{c}{\sc FBE} & \multicolumn{1}{c}{\sc NPG} & \multicolumn{1}{c|}{\sc Major} \\

\hline
      720 &    2560 &     160 &   0.1 &  1371 &  3596 &  3595 &   5.1 &   6.9 &   6.8 & 5.51199e-02 & 5.51702e-02 & 5.51662e-02\\
     1440 &    5120 &     320 &   0.7 &  1552 &  4148 &  4108 &  21.5 &  35.6 &  35.0 & 1.20602e-01 & 1.20660e-01 & 1.20663e-01\\
     2160 &    7680 &     480 &   0.7 &  1621 &  4203 &  4221 &  46.5 &  77.7 &  77.8 & 1.81414e-01 & 1.81472e-01 & 1.81473e-01\\
     2880 &   10240 &     640 &   1.4 &  1579 &  4326 &  4293 &  77.3 & 139.2 & 137.7 & 2.47298e-01 & 2.47356e-01 & 2.47359e-01\\
     3600 &   12800 &     800 &   2.5 &  1653 &  4479 &  4461 & 125.8 & 225.5 & 224.1 & 3.08809e-01 & 3.08873e-01 & 3.08874e-01\\
     4320 &   15360 &     960 &   3.8 &  1672 &  4484 &  4484 & 177.2 & 317.0 & 316.7 & 3.77171e-01 & 3.77235e-01 & 3.77235e-01\\
     5040 &   17920 &    1120 &   6.3 &  1757 &  4506 &  4602 & 251.7 & 433.8 & 442.7 & 4.39237e-01 & 4.39302e-01 & 4.39301e-01\\
     5760 &   20480 &    1280 &   8.2 &  1702 &  4528 &  4539 & 314.3 & 559.1 & 561.5 & 4.98815e-01 & 4.98878e-01 & 4.98882e-01\\
     6480 &   23040 &    1440 &  11.1 &  1708 &  4546 &  4539 & 401.5 & 719.2 & 716.2 & 5.69171e-01 & 5.69241e-01 & 5.69240e-01\\
     7200 &   25600 &    1600 &  15.1 &  1756 &  4578 &  4582 & 519.0 & 910.2 & 907.1 & 6.31546e-01 & 6.31609e-01 & 6.31614e-01\\
\hline
\end{tabular}
\end{footnotesize}
\end{table}

\begin{table}[h]
\caption{\small Results for random $A$ with unit column norms, $\mu = 10^{-3}$}
 \label{t2}
 \hspace{-0.8 cm}
\begin{footnotesize}
\begin{tabular}{|c c c ||r||r r r||r r r||r r r|}
\hline \multicolumn{3}{|c||}{size}& & \multicolumn{3}{c||}{iter} &
\multicolumn{3}{c||}{CPU} & \multicolumn{3}{c|}{fval}\\

\multicolumn{1}{|c}{ $m$} & \multicolumn{1}{c}{ $n$} &
\multicolumn{1}{c||}{ $s$} & {\sc t$_{\lambda_{\max}}$} & \multicolumn{1}{c}{\sc FBE} & \multicolumn{1}{c}{\sc NPG} & \multicolumn{1}{c||}{\sc Major} &
\multicolumn{1}{c}{\sc FBE} & \multicolumn{1}{c}{\sc NPG} & \multicolumn{1}{c||}{\sc Major} &
\multicolumn{1}{c}{\sc FBE} & \multicolumn{1}{c}{\sc NPG} & \multicolumn{1}{c|}{\sc Major} \\

\hline
     720 &    2560 &     160 &   0.1 &   898 &  2045 &  2054 &   3.2 &   3.7 &   3.7 & 1.16014e-01 & 1.16035e-01 & 1.16034e-01 \\
    1440 &    5120 &     320 &   0.7 &   966 &  2240 &  2225 &  13.1 &  18.6 &  18.4 & 2.45325e-01 & 2.45348e-01 & 2.45350e-01 \\
    2160 &    7680 &     480 &   0.6 &  1017 &  2291 &  2280 &  28.1 &  40.3 &  40.0 & 3.68869e-01 & 3.68896e-01 & 3.68897e-01 \\
    2880 &   10240 &     640 &   1.3 &  1017 &  2339 &  2327 &  48.6 &  72.7 &  72.0 & 4.94971e-01 & 4.94997e-01 & 4.94996e-01 \\
    3600 &   12800 &     800 &   2.4 &  1058 &  2398 &  2388 &  78.7 & 117.1 & 116.3 & 6.23388e-01 & 6.23416e-01 & 6.23418e-01 \\
    4320 &   15360 &     960 &   3.6 &  1057 &  2383 &  2384 & 109.2 & 163.6 & 163.1 & 7.50537e-01 & 7.50569e-01 & 7.50568e-01 \\
    5040 &   17920 &    1120 &   5.9 &  1051 &  2386 &  2376 & 147.5 & 223.0 & 221.8 & 8.69880e-01 & 8.69912e-01 & 8.69912e-01 \\
    5760 &   20480 &    1280 &   7.8 &  1064 &  2469 &  2433 & 193.8 & 296.7 & 293.5 & 1.01434e+00 & 1.01437e+00 & 1.01437e+00 \\
    6480 &   23040 &    1440 &  10.5 &  1092 &  2482 &  2448 & 251.7 & 380.2 & 375.8 & 1.13181e+00 & 1.13186e+00 & 1.13186e+00 \\
    7200 &   25600 &    1600 &  14.1 &  1084 &  2388 &  2409 & 313.4 & 458.8 & 461.6 & 1.25184e+00 & 1.25189e+00 & 1.25189e+00 \\
\hline
\end{tabular}
\end{footnotesize}
\end{table}

We also test some variants of {\bf FBE}$_{\rm L-BFGS}$ with a choice of $\gamma$ other than $0.95/L$. Specifically, we adopt the same algorithmic parameters as in {\bf FBE}$_{\rm L-BFGS}$ except that we consider three different $\gamma$'s: $0.5/L$, $0.7/L$ and $0.9/L$. These variants are denoted by FBE$_{0.5}$, FBE$_{0.7}$ and FBE$_{0.9}$ respectively. We use the same random instances from the previous experiment for $\mu = 10^{-3}$ in our test. The numerical results, averaged over the $10$ random instances for each $(m,n,s)$, are shown in Table~\ref{t3}. We see that FBE$_{0.9}$ takes the fewest number of iteration and the least CPU time, while it returns similar function values as FBE$_{0.5}$ and FBE$_{0.7}$. Moreover, comparing with Table~\ref{t2}, we see that {\bf FBE}$_{\rm L-BFGS}$ (whose $\gamma = 0.95/L$) is the fastest.

\begin{table}[h]
\caption{\small Results for FBE$_{0.5}$, FBE$_{0.7}$ and FBE$_{0.9}$ with $\mu = 10^{-3}$}
 \label{t3}
 \hspace{-0.8 cm}
\begin{footnotesize}
\begin{tabular}{|c c c ||r r r||r r r||r r r|}
\hline \multicolumn{3}{|c||}{size}& \multicolumn{3}{c||}{iter} &
\multicolumn{3}{c||}{CPU} & \multicolumn{3}{c|}{fval}\\

\multicolumn{1}{|c}{ $m$} & \multicolumn{1}{c}{ $n$} &
\multicolumn{1}{c||}{ $s$}  & \multicolumn{1}{c}{${\rm FBE_{0.5}}$} & \multicolumn{1}{c}{${\rm FBE_{0.7}}$} & \multicolumn{1}{c||}{${\rm FBE_{0.9}}$} &
\multicolumn{1}{c}{${\rm FBE_{0.5}}$} & \multicolumn{1}{c}{${\rm FBE_{0.7}}$} & \multicolumn{1}{c||}{${\rm FBE_{0.9}}$} &
\multicolumn{1}{c}{${\rm FBE_{0.5}}$} & \multicolumn{1}{c}{${\rm FBE_{0.7}}$} & \multicolumn{1}{c|}{${\rm FBE_{0.9}}$} \\

\hline
     720 &    2560 &     160 &  1266 &  1066 &   934 &   4.7 &   4.0 &   3.5 & 1.16014e-01 & 1.16014e-01 & 1.16014e-01\\
    1440 &    5120 &     320 &  1349 &  1135 &   993 &  18.6 &  15.6 &  13.6 & 2.45325e-01 & 2.45325e-01 & 2.45325e-01\\
    2160 &    7680 &     480 &  1434 &  1200 &  1052 &  41.1 &  34.4 &  30.1 & 3.68869e-01 & 3.68869e-01 & 3.68869e-01\\
    2880 &   10240 &     640 &  1459 &  1203 &  1049 &  71.6 &  59.0 &  51.4 & 4.94971e-01 & 4.94971e-01 & 4.94971e-01\\
    3600 &   12800 &     800 &  1500 &  1240 &  1085 & 115.1 &  95.2 &  83.3 & 6.23388e-01 & 6.23388e-01 & 6.23388e-01\\
    4320 &   15360 &     960 &  1499 &  1256 &  1096 & 160.2 & 133.8 & 116.8 & 7.50537e-01 & 7.50537e-01 & 7.50537e-01\\
    5040 &   17920 &    1120 &  1491 &  1234 &  1076 & 215.6 & 177.9 & 155.3 & 8.69880e-01 & 8.69880e-01 & 8.69880e-01\\
    5760 &   20480 &    1280 &  1522 &  1248 &  1090 & 281.9 & 231.5 & 202.1 & 1.01434e+00 & 1.01434e+00 & 1.01434e+00\\
    6480 &   23040 &    1440 &  1560 &  1294 &  1117 & 364.5 & 301.8 & 260.4 & 1.13181e+00 & 1.13181e+00 & 1.13181e+00\\
    7200 &   25600 &    1600 &  1542 &  1285 &  1120 & 452.6 & 378.0 & 329.0 & 1.25184e+00 & 1.25184e+00 & 1.25184e+00\\
\hline
\end{tabular}
\end{footnotesize}
\end{table}

Finally, we consider ill-conditioned problems to further evaluate the performance of {\bf FBE$_{\rm L-BFGS}$}, i.e., the matrix $A$ in \eqref{P-1} is ill-conditioned. Specifically,
as in \cite[Section~5]{YLHX14}, we let $A$ be a randomly over-sampled partial DCT matrix with columns given by
\begin{equation}\label{DCT}
  A_j = \frac1{\sqrt{m}}\cos(2\pi jw/F),~j = 1,...,n,
\end{equation}
where $w\in\R^m$ is a vector with independent entries uniformly sampled from $[0,1]$ and $F$ is a positive integer.

In our experiments below, we compare {\bf FBE$_{\rm L-BFGS}$} with {\bf NPG} for solving \eqref{P-1} on random instances where $A$ is generated by \eqref{DCT}. We first randomly generate a vector $x\in\R^n$ with sparsity $s$ and its entries in the support set are following i.i.d. standard Gaussian distribution. Then we let $b = Ax + \sigma\epsilon$, where $\sigma > 0$ and $\epsilon\in\R^m$ is a random vector with i.i.d. standard Gaussian entries.

In our numerical tests, for each $(m,n,s) = (100i,1500i,10ki)$ for $i=1,1.2$ and $k = 2,3,4$, we generate $30$ random instances as described above with $F = 20$ and $\sigma = 10^{-2}$. 
The computational results corresponding to \eqref{P-1} with $\mu = 10^{-4}$ are presented in Table~\ref{t_dct}, where we denote the algorithm {\bf FBE$_{\rm L-BFGS}$} with $tol = 10^{-6}$ by FBE, the algorithm {\bf NPG} terminated at $tol = 10^{-6}$ and $tol = 10^{-5}$ by {\bf NPG$_{-6}$} and {\bf NPG$_{-5}$} respectively. In the table, we report the time for computing $\lambda_{\max}(A^TA)$ ({\sc t}$_{\lambda_{\max}}$), the number of iterations (iter), CPU time in seconds (CPU) and the terminating function values (fval), averaged over the $30$ random instances. We see that {\bf FBE$_{\rm L-BFGS}$} always outperforms the {\bf NPG}$_{-6}$ in terms of the number of iterations, CPU time and function values. On the other hand, while {\bf NPG}$_{-5}$ takes the fewest iterations and least CPU time, it has much worse function values, indicating that the termination is likely premature.

\begin{table}[h]
\caption{\small Results for randomly over-sampled partial DCT matrix $A$, $\mu = 10^{-4}$}
 \label{t_dct}
 \hspace{-0.8 cm}
\begin{footnotesize}
\begin{tabular}{|c c c ||r||r r r||r r r||r r r|}
\hline \multicolumn{3}{|c||}{size}& & \multicolumn{3}{c||}{iter} &
\multicolumn{3}{c||}{CPU} & \multicolumn{3}{c|}{fval}\\

\multicolumn{1}{|c}{ $m$} & \multicolumn{1}{c}{ $n$} &
\multicolumn{1}{c||}{ $s$} & {\sc t$_{\lambda_{\max}}$} & \multicolumn{1}{c}{\sc FBE} & \multicolumn{1}{c}{${\rm NPG_{-6}}$} & \multicolumn{1}{c||}{${\rm NPG_{-5}}$} &
\multicolumn{1}{c}{\sc FBE} & \multicolumn{1}{c}{${\rm NPG_{-6}}$} & \multicolumn{1}{c||}{${\rm NPG_{-5}}$} &
\multicolumn{1}{c}{\sc FBE} & \multicolumn{1}{c}{${\rm NPG_{-6}}$} & \multicolumn{1}{c|}{${\rm NPG_{-5}}$} \\

\hline
     100 &    1500 &      20 &   0.0 & 3e+04 & 2e+05 & 1e+04 &  18.7 &  25.8 &   1.6 & 1.5844e-03 & 1.5873e-03 & 1.8627e-03\\
     100 &    1500 &      30 &   0.0 & 3e+04 & 2e+05 & 1e+04 &  18.9 &  30.7 &   1.9 & 2.1274e-03 & 2.1325e-03 & 2.4659e-03\\
     100 &    1500 &      40 &   0.0 & 3e+04 & 3e+05 & 2e+04 &  21.3 &  36.7 &   2.3 & 2.7476e-03 & 2.7521e-03 & 3.1333e-03\\
     120 &    1800 &      24 &   0.0 & 3e+04 & 2e+05 & 1e+04 &  25.2 &  36.3 &   2.4 & 2.0271e-03 & 2.0306e-03 & 2.3160e-03\\
     120 &    1800 &      36 &   0.0 & 4e+04 & 3e+05 & 2e+04 &  32.9 &  46.6 &   2.7 & 2.6018e-03 & 2.6053e-03 & 2.9544e-03\\
     120 &    1800 &      48 &   0.0 & 3e+04 & 3e+05 & 2e+04 &  31.2 &  49.2 &   3.4 & 3.0467e-03 & 3.0505e-03 & 3.4297e-03\\
\hline
\end{tabular}
\end{footnotesize}
\end{table}

\section{Concluding remarks}\label{sec:conclusion}

In this paper, we further studied the forward-backward envelope first introduced in \cite{FBE14} and \cite{FBE16}, and established sufficient conditions for the envelope to be a level-bounded and KL function with an exponent of $\frac12$. We also illustrated how to solve a class of difference-of-convex regularized least squares problem via a suitable forward-backward envelope. This opens up the possibility of applying techniques for smooth unconstrained optimization to this special class of difference-of-convex programming problems.

One open question is how to extend the concept of forward-backward envelope to \eqref{P0} in the case when $f$ only has locally Lipschitz continuous gradients; this instance also arises frequently in applications. Another open question is to identify more classes of (nonconvex) problems that satisfy the error bound condition (Assumption~\ref{assum1}). In particular, it is still unknown to us whether the objective function in \eqref{P4} satisfies the error bound condition.

\appendix
\section{Closed form formula for NPG subproblems}\label{AppendixA}

In this appendix, we derive a closed form formula for the following problem,
\begin{equation}\label{l1l2subproblem}
  \min_{x} \ \frac{1}{2}\|x - y\|^2 + \mu_1 \|x\|_1 - \mu_2 \|x\|,
  \end{equation}
where $\mu_1 \ge \mu_2 > 0$ and $y\in \R^n$ is given. It is easy to see that \eqref{l1l2subproblem} covers \eqref{subproblem2} as a special case and hence we will obtain a closed form formula
for these NPG subproblems. To proceed with our derivation, we first establish the following lemma.
\begin{lemma}\label{A_lem1}
  Let $v\in \R^n$ and define ${\cal I} := \{i:\; v_i < 0\}$. Consider the optimization problem
  \begin{equation}\label{prelim}
  \min_{\|x\|= 1, x\ge 0} v^Tx
  \end{equation}
  \begin{enumerate}[{\rm (i)}]
    \item Suppose that ${\cal I}\neq \emptyset$. Then an optimal solution $x^*$ of \eqref{prelim} is given by
    \begin{equation}\label{xform1}
    x^*_i = \begin{cases}
      -\frac{v_i}{\|v_{{\cal I}}\|} & {\rm if}\ i \in {\cal I},\\
      0 & {\rm otherwise},
    \end{cases}
    \end{equation}
    where $v_{\cal I}$ is the subvector of $v$ indexed by ${\cal I}$.
    \item Suppose that ${\cal I} = \emptyset$ and take an $i_*\in \{i:\; v_i = \min_k v_k\}$. Then an optimal solution $x^*$ of \eqref{prelim} is given by
    \begin{equation}\label{xform2}
      x^*_i = \begin{cases}
        1 & {\rm if}\ i = i_*,\\
        0 & {\rm otherwise}.
      \end{cases}
    \end{equation}
  \end{enumerate}
\end{lemma}
\begin{proof}
  It is clear that an optimal solution of \eqref{prelim} exists.

  Suppose first that ${\cal I} \neq \emptyset$. In this case, we consider the following relaxation of \eqref{prelim}:
  \begin{equation}\label{prelim_relax}
  \min_{\|x\|\le 1, x\ge 0} v^Tx
  \end{equation}
  Let $x^*$ be an optimal solution of \eqref{prelim_relax}. Then it is easy to see that
  for any $i$ that corresponds to $v_i > 0$, we must have $x^*_i = 0$; because otherwise, one can zero out these $x_i^*$ to obtain a feasible solution
  with a strictly smaller objective value. Next, since ${\cal I}$ is nonempty,
  it is not hard to see that one must have $x^*_i = 0$ for all $i$ corresponding to $v_i = 0$, because otherwise, one can further decrease the objective value by setting these entries to zero while increasing some $x_i^*$ with $i\in {\cal I}$. Thus, $x^*_i = 0$ for all $i$ that correspond to $v_i \ge 0$. Finally, it must hold that $\|x^*\| = 1$ because otherwise one can further increase $x_i^*$ for $i\in {\cal I}$ to decrease the objective value. Thus, we conclude that $x^*$ must take the form of \eqref{xform1}. Since this $x^*$ is optimal for \eqref{prelim_relax}
  and is also feasible for \eqref{prelim}, it must also be optimal for \eqref{prelim}.

  Next, suppose that ${\cal I}$ is empty. This means that $v$ is a nonnegative vector. Observe that for any $x$ feasible for \eqref{prelim} so that $\|x\| = 1$, we have
  \begin{equation*}
    1 = \|x\|^2 \le (e^Tx)^2 \le n \|x\|^2 = n,
  \end{equation*}
  showing that the following optimization problem is a relaxation of \eqref{prelim}:
  \begin{equation}\label{prelim_relax2}
  \min_{1\le e^Tx \le \sqrt{n}, x\ge 0} v^Tx.
  \end{equation}
  Since $v$ is nonnegative, a simple scaling argument indicates that any optimal solution $x^*$ of \eqref{prelim_relax2} has to satisfy $e^Tx^* = 1$. In particular, this shows that the optimal value of \eqref{prelim_relax2} is given by $\min_i v_i$ and the $x^*$ given by \eqref{xform2} is an optimal solution. Since this $x^*$ is clearly feasible for \eqref{prelim}, it is also optimal for \eqref{prelim}. This completes the proof.
\end{proof}

The next proposition gives an explicit formula for a minimizer of \eqref{l1l2subproblem}.

\begin{proposition}
  Let ${\cal I} = \{i:\; \mu_1 < |y_i|\}$.
  \begin{enumerate}[{\rm (i)}]
    \item Suppose that ${\cal I}$ is nonempty. Then a solution $x^*$ of \eqref{l1l2subproblem}
  is given by
  \begin{equation*}
    x^*_i =
    \begin{cases}
    {\rm sgn}(y_i)(\mu_2 + \||y_{\cal I}|-\mu_1 e_{\cal I}\|)\frac{|y_i| - \mu_1}{\||y_{\cal I}|-\mu_1 e_{\cal I}\|} & {\rm if}\ i \in {\cal I},\\
    0 & {\rm otherwise},
    \end{cases}
  \end{equation*}
  where $y_{\cal I}$ is the subvector of $y$ indexed by ${\cal I}$, the absolute value $|y_{\cal I}|$ is taken componentwise, and $e_{\cal I}$ is the vector of all ones of dimension $|{\cal I}|$.
  \item Suppose that ${\cal I}$ is empty and take an $i_*\in \{i:\; \mu_1 - |y_i| = \min_k\{\mu_1 - |y_k|\}\}$. Then a solution $x^*$ of \eqref{l1l2subproblem}
  is given by
  \begin{equation*}
      x^*_i = \begin{cases}
        {\rm sgn}(y_{i_*})\max\{\mu_2 - (\mu_1-|y_{i_*}|),0\} & {\rm if}\ i = i_*,\\
        0 & {\rm otherwise}.
      \end{cases}
    \end{equation*}
  \end{enumerate}
%
%
\end{proposition}
\begin{proof}
  Using a substitution $x = \alpha\circ w$, where $\alpha \in \{-1,1\}^n$, $w\ge 0$ and $\circ$ denotes entrywise product, and expanding the quadratic term,
  we see that problem \eqref{l1l2subproblem} can be equivalently
  written as
  \begin{equation*}
    \min_{\alpha \in \{-1,1\}^n,w\ge 0} \frac{1}{2}\|w\|^2 - (\alpha\circ y)^Tw + \mu_1 e^Tw - \mu_2 \|w\|.
  \end{equation*}
  Applying a further substitution $w = ru$ with a number $r\ge 0$ and a nonnegative vector $\|u\|= 1$, the above problem can be further reformulated as
  \begin{equation*}
    \min_{
    \begin{subarray}\
    \alpha \in \{-1,1\}^n\\
    u\ge 0,\|u\|= 1
    \end{subarray}
    }\min_{r\ge 0} \frac{1}{2}r^2 - r(\alpha\circ y)^Tu + \mu_1r e^Tu - \mu_2 r.
  \end{equation*}
  It is easy to check that the inner minimization is
  attained at $r = \max\{\mu_2 - (\mu_1e-\alpha\circ y)^Tu,0\}$. Plugging this back, the optimization problem
  now becomes
  \begin{equation*}
    \min_{
    \begin{subarray}\
    \alpha \in \{-1,1\}^n\\
    u\ge 0,\|u\|= 1
    \end{subarray}
    }-\frac{1}{2}(\max\{\mu_2 - (\mu_1e-\alpha\circ y)^Tu,0\})^2.
  \end{equation*}
  Since the function $t\mapsto -\frac{1}{2}(\max\{\mu_2 - t,0\})^2$ is nondecreasing,
  to obtain an optimal solution of the above optimization problem, one only needs to consider the problem
  \begin{equation*}
    \min_{
    \begin{subarray}\
    \alpha \in \{-1,1\}^n\\
    u\ge 0,\|u\|= 1
    \end{subarray}
    }(\mu_1e-\alpha\circ y)^Tu.
  \end{equation*}
  For this problem, since $u\ge 0$, we must have $\alpha^* = {\rm sgn}(y)$ at optimality. This further reduces the above problem to
  \begin{equation*}
    \min_{u\ge 0,\|u\|=1}(\mu_1e-|y|)^Tu.
  \end{equation*}
  The conclusion of this proposition now follows from this observation, Lemma~\ref{A_lem1}, the facts that $x = \alpha\circ(ru)$ with $r = \max\{\mu_2 - (\mu_1e-\alpha\circ y)^Tu,0\}$ and $\alpha^* = {\rm sgn}(y)$.
\end{proof}

\end{document}